\documentclass[11pt,oneside,final]{amsart}
\usepackage{stmaryrd}
\title{An inverse free boundary problem}

\author[C. I. C\^{a}rstea]{C\u{a}t\u{a}lin I. C\^{a}rstea}
\address{Department of Applied Mathematics, National Yang Ming Chiao Tung University, Hsinchu, Taiwan}
\email{catalin.carstea@gmail.com}

\author[M. Lassas]{Matti Lassas}
\address{Department of Mathematics and Statistics, University of Helsinki, P.O. Box 68 (Pietari Kalmin katu 5), FI-00014 University of Helsinki, Finland}
\email{matti.lassas@helsinki.fi}

\author[J. Lu]{Jinpeng Lu}
\address{School of Mathematical Sciences, University of Science and Technology of China, Hefei 230026, China}
\email{jinpeng.lu@ustc.edu.cn}

\author[L. Oksanen]{Lauri Oksanen}
\address{Department of Mathematics and Statistics, University of Helsinki, P.O. Box 68 (Pietari Kalmin katu 5), FI-00014 University of Helsinki, Finland}
\email{lauri.oksanen@helsinki.fi}

\author[Z. Zhao]{Ziyao Zhao}
\address{Department of Mathematics and Statistics, University of Helsinki, P.O. Box 68 (Pietari Kalmin katu 5), FI-00014 University of Helsinki, Finland}
\email{ziyao.zhao@helsinki.fi}

\IfFileExists{tweakslo.sty}{\usepackage{tweakslo}}{
\usepackage{amssymb,thmtools,mathtools,todonotes,amsmath}
\declaretheorem{theorem,definition,lemma,proposition,corollary,remark}}

\usepackage{tikz}
\usetikzlibrary{arrows.meta,calc}

\def\p{\partial}
\def\o2{\overline{O_2}}

\def\RR{\mathbb R}

\def\BB{\mathbb B}

\def\cK{\mathcal K}
\def\cKpar{\mathcal{K}^{t}}

\def\cH{\mathcal H}
\def\cA{\mathcal A}
\def\cApar{\mathcal{A}^{t}}

\def\tphi{\widetilde{\varphi}}
\def\tM{\widetilde{M_0}}

\def\dist{\text{dist}}

\def\wv{\widetilde{v}}

\def\Lambdapar{\Lambda^{t}}

\def\Cap{\operatorname{cap}}
\DeclareMathOperator{\supp}{supp}
\def\limess{\qopname\relax m{lim\,ess}}
\DeclarePairedDelimiter{\norm}{\lVert}{\rVert}
\DeclarePairedDelimiter{\abs}{\lvert}{\rvert}

\date{\today}

\begin{document}
\begin{abstract}
We study inverse problems for the elliptic and parabolic obstacle problems from boundary measurements. For the classical elliptic obstacle problem with strictly superharmonic obstacle function, we show that the Dirichlet-to-Neumann map admits a one-sided linearization at every boundary datum lying strictly above the obstacle. The linearized map is the Dirichlet-to-Neumann map for a rough Dirichlet problem on the a priori unknown non-contact set. We show that these linearized Cauchy data uniquely determine the non-contact set up to set of Sobolev $2$-capacity zero and consequently determine the obstacle.
Our result applies to the inverse problem for a parabolic obstacle problem where both the coefficient and the obstacle function are time-independent by reducing to the elliptic inverse problem.
\end{abstract}
\maketitle

\tableofcontents

\section{Introduction}

Variational inequalities arise from the minimization of energy functionals over a convex set of constraints in the calculus of variations, giving rise to a system of differential inequalities instead of the classical Euler-Lagrange equations. 
The systematic study of variational inequalities began in the 1960s motivated by the Signorini problem in elasticity \cite{Signorini59,Fichera64} and by questions in potential theory \cite{Stampacchia64}, and was later developed into a general theory in \cite{LS67,Brezis68}.
A characteristic feature of such constrained problems is that they produce free boundaries: the a priori unknown interface surrounding the domain where the constraint is active. 
As the prototypical example of this class, the obstacle-type problems appear in several fundamental models across mechanics, fluid dynamics, biology and finance, see e.g. \cite{DL76,Chipot84, Rodrigues1987,KO88,F-survey}.

\subsection{Elliptic obstacle problem}
Let $\Omega\subset \mathbb R^n$ be a bounded connected open set with smooth boundary and $\varphi\in C^{\infty}(\overline\Omega)$ modeling the obstacle.
We denote the set of smooth Dirichlet boundary values that are compatible with the obstacle $\varphi$ by $\cA_{\varphi}$, that is
\begin{equation}
    \cA_{\varphi}: = \{f\in C^\infty(\p \Omega) \mid f \geq \varphi|_{\p \Omega}\}.
\end{equation}
In the classical formulation of the obstacle problem,
one seeks to minimize the Dirichlet energy
\begin{equation} \label{def-energy-intro}
E(u)=\int_{\Omega} |\nabla u(x)|^2\,dx
\end{equation}
among all functions in the convex set
\begin{equation}\label{def-K-intro}
\cK_{\varphi,f}=\{v\in H^1(\Omega): v\geq \varphi \text{ a.e. in }\Omega,\ v|_{\partial \Omega}=f\},
\end{equation}
given the compatibility condition $f\in \cA_\varphi$.
Then the calculus of variations yields the variational inequality for minimizer $u$,
\begin{equation}\label{def-variational-intro}
\int_{\Omega} \nabla u\cdot \nabla (v-u)\,dx \geq 0,
\qquad \text{for all } v\in \cK_{\varphi,f}.
\end{equation}
This variational problem has a unique minimizer by \cite{LS67}, and the optimal regularity of the minimizer is of $C^{1,1}_{\text{loc}}(\Omega)$ by \cite{Frehse1972}.
Thus the minimizer $u$ solves
\begin{equation}
    \label{def-obstacle-intro}
    \begin{cases}
      u\geq \varphi & \textrm{ in } \Omega, \\
      \Delta u \leq 0 & \textrm{a.e. in } \Omega, \\
      \Delta u=0 &\textrm{ in } \{x : u(x) > \varphi(x)\}, \\
      u|_{\partial \Omega} =f,
    \end{cases}
  \end{equation}
or, equivalently,
\begin{equation}
    \label{def-obstacle-intro-min}
      \min\{-\Delta u, u-\varphi\}=0 \quad \textrm{a.e. in }\Omega,
  \end{equation}
subject to the Dirichlet boundary condition $u|_{\partial \Omega}=f$. 
We consider the following inverse problem for the classical obstacle problem.

\medskip
\noindent\textbf{Inverse obstacle problem.} \emph{
Is it possible to determine the obstacle $\varphi$ from boundary measurements made for the minimizers?}
\medskip



The main difficulty in this type of inverse problems is that the constraint may be hidden from the observable boundary data. In the classical obstacle problem, the obstacle condition is only activated in the contact set $\{x\in \Omega:\ u(x)=\varphi(x)\}$ which is unknown a priori.
This implicit dependence of the contact set on the solution itself poses a fundamental challenge for both the regularity theory of the free boundary and the analysis of inverse problems \cite{Figalli-ICM2018,Ros2018-survey,HKS26}.
In an early work \cite{lewy1979}, Lewy showed that the obstacle $\varphi$ can be reconstructed given a prescribed family of contact sets for the obstacles $\varphi+c$ with constants $c>0$. For our inverse problem with boundary measurements, the contact sets are unknown and thus need to be recovered rather than assumed.

Specifically for our inverse obstacle problem, observe that the superharmonicity of obstacles is necessary for unique determination.
Namely, if one allows the obstacle to have $\Delta \varphi (x_0)>0$ at some point $x_0\in \Omega$, then the obstacle condition is not active in some neighborhood of $x_0$ by the strong maximum principle, and thus a small perturbation of the obstacle near $x_0$ does not affect at all the boundary measurements of solutions (for any given boundary value).
Even for superharmonic obstacles, any finite number of boundary measurements cannot uniquely determine the obstacle. 
This can be seen by picking a small ball $B$ in the complement of the union of the free boundaries for any finite number of given boundary values, and perturbing the obstacle function in $B$.
Indeed, if the small ball $B$ is contained in the contact set for a given boundary value, then perturbing the obstacle $\varphi$ to $\varphi-h$, where $h\in C_0^{\infty}(B)$ with $h\geq 0$ and $\Delta h > \Delta \varphi$, does not change the solution on $\Omega\setminus B$ by the maximum principle.
If $B$ is contained in the non-contact set for a given boundary value, then the obstacle condition is not active in $B$ so perturbations in $B$ are invisible to the boundary measurements.


In this paper, we consider the Dirichlet-to-Neumann map
\begin{align}
\label{eq:def_elliptic_DN}
    \Theta_\varphi : \cA_\varphi \to C(\p\Omega),\\
    \Theta_\varphi(f)=\partial_\nu u_\varphi^f\big|_{\partial\Omega},
\end{align}
where $u_\varphi^f$ is the solution to \eqref{def-obstacle-intro} associated with the obstacle $\varphi$ and the Dirichlet boundary value $f$.
Our first main theorem is stated as follows.

\smallskip
\begin{theorem} \label{thm-main-1}
Let $\Omega\subset \RR^n$ be a bounded open set with smooth boundary. Suppose the obstacles $\varphi_1,\varphi_2\in C^{2,1}(\overline\Omega)$ satisfy
\[
\varphi_1|_{\p\Omega}=\varphi_2|_{\p\Omega},
\qquad
\Delta\varphi_1<0,\ \Delta\varphi_2<0 \quad\text{in }\Omega.
\]
Let $\Theta_{\varphi_1}, \Theta_{\varphi_2}$ be the corresponding elliptic Dirichlet-to-Neumann maps defined in \eqref{eq:def_elliptic_DN} associated with the obstacle $\varphi_1, \varphi_2$, respectively. 
If $\Theta_{\varphi_1} = \Theta_{\varphi_2}$, then $\varphi_1=\varphi_2$ in $\Omega$.
\end{theorem}

\smallskip

Our idea of proof is as follows. After subtracting the common harmonic lift of the boundary values, the problem is reduced to obstacles vanishing on the outer boundary. 
We show that the nonlinear Dirichlet-to-Neumann map admits a left linearization at every strictly positive boundary datum, and that the derivative is the Dirichlet-to-Neumann map of a Dirichlet problem on the rough unknown non-contact set. At the level of this linearized problem we prove that the outer Cauchy data uniquely determine the non-contact set up to sets of Sobolev $2$-capacity zero.
The notion of Sobolev $2$-capacity arises naturally in this context. Owing to the roughness of the non-contact set, the optimal regularity expected for the solution to the associated Dirichlet problem is merely $H^1$, whose pointwise behavior is only defined quasi-everywhere, that is, up to sets of Sobolev $2$-capacity zero.
Finally, by varying scalar multiples of a fixed positive boundary value and integrating the resulting one-parameter family of left linearizations, we recover the obstacle itself.

\subsection{Parabolic obstacle problem}
We next consider a parabolic obstacle problem \cite{ACM-parabolic} on $[0,\infty)\times \Omega$ associated with the operator $\partial_t-c(x)\Delta$
and a time-independent obstacle function $\varphi=\varphi(x)$, where the scalar coefficient $c(x)>0$ is smooth and time-independent. Let $M_0$ be a constant with $\varphi(x)<M_0$ in $\Omega$, and assume that the boundary value $\varphi|_{\partial\Omega}$ is known. We define the admissible class of Dirichlet data by
\begin{equation}
  \cApar_{\varphi,M_0}:=
  \{f\in C^\infty([0,\infty)\times \partial\Omega):
  f(0,\cdot)=M_0,\ f(t,x)>\varphi(x)|_{\partial\Omega} \textrm{ for all }t\}.
\end{equation}
For $f\in \cApar_{\varphi,M_0}$, we define the convex set 
  \begin{equation}
    \cKpar_{\varphi,f} : = \{u\in H^1([0,\infty)\times \Omega)\mid u(0,\cdot) = M_0,\ u|_{[0,\infty)\times \p\Omega} = f|_{[0,\infty)\times \p\Omega},\ u\geq \varphi\ a.e.\}.
  \end{equation}
The parabolic variational inequality problem with obstacle $\varphi$ amounts to finding $u\in \cKpar_{\varphi,f}$ satisfying
\begin{equation}
\label{eq:par_obstacle_weak}
    \int_{\Omega}\p_t u (t,x) (v(t,x)-u(t,x)) + \nabla u(t,x)\nabla(c(x)(v(t,x)-u(t,x)))\, dx\geq 0,\ a.e.\ t\in [0,\infty).
\end{equation}
We denote
\[
W^{1,2}_p([0,T]\times \Omega)
:=
\{u: D_x^2u,\ D_xu,\ \p_t u\in L^p([0,T]\times \Omega)\}.
\]
In view of \cite[Theorem 8.2]{Friedman82}, the solution $u\in W^{1,2}_p([0,T]\times \Omega)$ for every $T>0$ and $1<p<\infty$. Thus $u$ solves
\begin{equation}
  \label{eq:par_obstacle_strong}
  \begin{cases}
    (\partial_t-c(x)\Delta)u(t,x)\geq 0,&\text{a.e. in }[0,\infty)\times\Omega,\\
    u(t,x)\geq \varphi(x),&\text{a.e. in }[0,\infty)\times\Omega, \\
    (\partial_t u(t,x)-c(x)\Delta u(t,x))\cdot (u(t,x)-\varphi(x))=0,&\text{a.e. in }[0,\infty)\times\Omega,\\
    u|_{[0,\infty)\times \partial\Omega}=f,\qquad u(0,\cdot)=M_0.
  \end{cases}
\end{equation}
Then we define the nonlinear parabolic Dirichlet-to-Neumann map by
\begin{align}
  \label{eq:def_Lambda}
  \Lambdapar_{c,\varphi}:\cApar_{\varphi,M_0}&\to C([0,\infty)\times \partial\Omega),\\
  \Lambdapar_{c,\varphi}(f)&=\partial_\nu u_\varphi^f|_{[0,\infty)\times \partial\Omega}.
\end{align}
The inverse parabolic obstacle problem is to determine the obstacle $\varphi$ and the coefficient $c$ from the parabolic Dirichlet-to-Neumann map.

The first step of solving the inverse problem is the coefficient recovery. Since the constant state $M_0$ lies strictly above the obstacle, sufficiently small perturbations around $M_0$ do not activate the obstacle, and hence the Dirichlet-to-Neumann map for the parabolic obstacle problem reduces to the Dirichlet-to-Neumann map for the linear equation $(\partial_t-c(x)\Delta)u=0$. This allows us to recover the coefficient $c$ through the classical theory of inverse problems. The second step is a reduction to the inverse elliptic obstacle problem: once the boundary values become stationary, the solution for the parabolic obstacle problem converges to the solution for the corresponding elliptic obstacle problem as $t\to\infty$, and the equality of parabolic Dirichlet-to-Neumann maps implies the equality of the induced elliptic Dirichlet-to-Neumann maps.

Our second main result of the paper is stated as follows.


\smallskip
\begin{theorem}
  \label{thm:main}
  Let $\Omega\subset \RR^n$ be a bounded open set with smooth boundary.
  Let $\varphi_1,\varphi_2\in C^{2,1}(\overline\Omega)$ satisfy
\[
\varphi_1|_{\p\Omega}=\varphi_2|_{\p\Omega},
\qquad
\Delta\varphi_1<0,\ \Delta\varphi_2<0 \quad\text{in }\Omega,
\]
and $\varphi_1,\varphi_2<M_0$ for some constant $M_0$. Let $c_1,c_2\in C^\infty(\overline\Omega)$ be strictly positive. 
  Suppose that $\Lambdapar_{c_1,\varphi_1}=\Lambdapar_{c_2,\varphi_2}$.
  Then $c_1=c_2$ and $\varphi_1=\varphi_2$ in $\Omega$.
\end{theorem}

\subsection{Literature review}
The classical inverse obstacle problem studies the detection of unknown inclusions or cavities from their scattering signatures under acoustic, electromagnetic or gravitational probing fields \cite{LP67,Isakov-survey,CK18}. A parallel program, initiated by Isakov \cite{Isakov88} and subsequently advanced by Ikehata \cite{Ikehata98}, shifted the focus from pure shape reconstruction to coupling the unknown inclusion with the conductivity of the surrounding medium, seeking to recover the internal structure from the Dirichlet-to-Neumann map. This approach, fundamentally linked to Calderón's problem, was further propelled by the development of complex geometric optics \cite{SU87,NUW05,IINSU07,UW07,UWW09}. 
In the case of single measurement, establishing global uniqueness inherently requires restrictive geometric assumptions on the obstacles, such as convex polyhedra \cite{FI89} or finite unions of disks \cite{IP90}. Later in \cite{AI96} a key observation linked the uniqueness in the obstacle determination to the local regularity of the obstacle, framing it as a free boundary within a transmission problem. This perspective was rigorously expanded by Athanasopoulos, Caffarelli and Salsa \cite{ACS01} to accommodate a broader class of obstacles. 
Recently in \cite{SS21,SS25}, methods from free boundary problems were employed by Salo and Shahgholian to study problems arising from inverse scattering theory. We also mention a recent result on the recovery of a potential for a fractional Schr\"odinger operator with an obstacle \cite{UZ26}.

Free boundary problems are of broad interest due to their diverse applications in physics and engineering, such as fluid filtration through porous media \cite{DL76}, phase transitions \cite{Stefan1889}, optimal stopping problems for stochastic processes \cite{HP89}. 
The vast majority of the literature has focused on the regularity of the solution \cite{Shahgholian03,ACM-parabolic} and the free boundary \cite{KN77,Caffarelli78,FRS20,FS19}, see \cite{CS05,PetrosyanShahgholianUraltseva2012,Figalli-ICM2018,Ros2018-survey,FR22} and the references therein. 
Quantitative stability for the solution and the free boundary under perturbations of the obstacle was studied in \cite{SS18}.
In contrast, the study of inverse problems for free boundary models has been remarkably limited.
Inverse problems for the Signorini (thin) obstacle problem were recently studied in \cite{DLLOZ2025,Z26}.
For the classical (thick) obstacle problem with a known obstacle, it is shown in \cite{HKS26} that the recovery of the scalar coefficient with interior measurements is non-unique in the contact set. To the best of our knowledge, our present work provides the first uniqueness result for the recovery of the unknown obstacle itself in the classical obstacle problem from boundary measurements.
More broadly, inverse problems for the inclusion-type models have been studied extensively for parabolic equations. Multiple reconstruction methods have been developed, including the exterior approach \cite{BD-heat}, the enclosure method \cite{IK-heat,IK-heat-reconstruction}, 
the probe method \cite{DKN,IKN}, and sampling method \cite{NW-heat,SNW-heat}, while uniqueness and stability for the inclusion-type problem were established in \cite{EI-parabolic-uniqueness,DV-parabolic-stability}. 
We refer the reader to the survey \cite{Isakov-survey} for a broader overview of inverse obstacle problems. 
In a related work \cite{Kian2020}, the simultaneous determination of coefficients, sources and an obstacle for a parabolic equation from a single measurement is studied.

There is a substantial literature on the inverse coefficient problems for linear parabolic equations. 
For time-independent coefficients we refer to \cite{Isakov1993,CK01}; for time-dependent lower-order terms and partial-data problems, see \cite{CK2018,FKU2024}; and the case of nonautonomous heat equations with time-dependent coefficients is considered in \cite{Feizmohammadi2024}.
From a methodological point of view, a central approach is to reduce the parabolic inverse problem to spectral or wave-type data \cite{Belishev-heat,KKLM2004,KKL}. 
When the potential is time-independent, this reduction can be combined with the uniqueness of inverse spectral problems \cite{NSU1988,Isozaki91,CS13,KMO19}. In geometric or anisotropic settings, the boundary control method provides a powerful and general framework \cite{Belishev1987,BK1992,KKL}. 


\medskip
\noindent {\bf Outline of the paper.}
The paper is organized as follows. 
In Section \ref{sec:rounge_Dirichlet} we recall some basic facts on the classical obstacle problem and study a Dirichlet problem on the rough non-contact set. In particular, we prove that the one-sided linearization of the obstacle problem can be characterized by the rough Dirichlet problem on the corresponding non-contact set. Section \ref{sec:inverse_elliptic} is devoted to solving the inverse elliptic obstacle problem stated in Theorem \ref{thm-main-1}: we show that the linearized Cauchy data determine the non-contact set up to set of Sobolev $2$-capacity zero, and then integrate the resulting one-parameter family in the boundary amplitude to recover the obstacle. In Section \ref{sec:inverse_parabolic} we establish the regularity and asymptotic convergence needed to reduce the inverse parabolic obstacle problem to the elliptic inverse problem, and then conclude Theorem \ref{thm:main} from the reduction.


\section{A Rough Dirichlet Problem for the Elliptic Obstacle}
\label{sec:rounge_Dirichlet}

\subsection{Preliminaries}
We start by recalling basic definitions and facts in the classical obstacle problem.
Let $\varphi\in C^{2,1}(\overline{\Omega})$ with the obstacle satisfying $\Delta \varphi < 0$. Let $u^g$ solve \eqref{def-obstacle-intro} associated with the boundary condition $g\in C^{\infty}(\p \Omega)$ satisfying $g > \varphi|_{\p\Omega}$. In view of the global regularity for the Dirichlet obstacle problem (see e.g. \cite{Jensen1980,Friedman82}), there exists a constant $C=C(\alpha,\Omega,\varphi)$ such that
  \begin{equation}
    \label{eq:ell_c1alpha_estimate}
    \|u^g\|_{C^{1,\alpha}(\overline\Omega)}
    \leq C\bigl(1+\|g\|_{C^{2,\alpha}(\p\Omega)}\bigr).
  \end{equation}

We define the non-contact set $H_g$ and contact set $K_g$ as
\[
H_g:=\{x\in\Omega:u^g(x)>\varphi(x)\},
\qquad
K_g:=\{x\in\Omega:u^g(x)=\varphi(x)\}.
\]

\begin{lemma}
  \label{lem:ell_no_interior_components}
  Let $A$ be any connected component of $H_g$. Then $\overline{A}\cap \p \Omega\neq \emptyset$.
\end{lemma}
\begin{proof}
To get a contradiction, we assume that $\overline{A}\cap \p \Omega = \emptyset$.
Since $A\subset H_g$ and $\p A\cap \p \Omega = \emptyset$, then $u^g$ is harmonic in $A$ and $u^g=\varphi$ on $\p A$.
Hence $w:=u^g-\varphi$ satisfies
\[
\Delta w=-\Delta \varphi>0\quad\text{in }A,
\qquad
w=0\quad\text{on }\p A.
\]
Since $\Delta w>0$ in $A$, the function $w$ is a strict subsolution of the Laplace equation. If there were a point in $A$ with $w>0$, then, because $w=0$ on $\p A$, the continuous function $w$ would attain a positive maximum at an interior point of $A$. By the strong maximum principle \cite[Theorem 3.5]{gilbarg1998}, this would force $w$ to be constant in $A$, which is impossible since $\Delta w>0$. Hence $w\leq 0$ in $A$, contradicting $A\subset H_g$.
\end{proof}

\begin{lemma}
\label{lm:free_db_density}
    Let $g\in C^\infty(\p \Omega)$ with $g>\varphi|_{\p\Omega}$. If $x\in \p H_g\setminus \p \Omega$, then there exist constants $r_0, \delta>0$, such that for all $0<r\leq r_0$, there holds
    \begin{equation}
        \frac{\abs{H_g\cap B(x,r)}}{\abs{B(x,r)}} \geq \delta >0.
    \end{equation}
\end{lemma}
\begin{proof}
    Since $g > \varphi|_{\p\Omega}$ and $u^g\in C^{1}(\Omega)$, there exists a smooth open collar neighborhood $U$ of $\p\Omega$ such that $U\subset H_g$ and $x\not\in \overline{U}$. Notice that $\Delta \varphi <0$ in $\Omega$, then $\Delta \varphi$ is strictly negative on the compact set $\Omega\setminus U$, that is, we can find a constant $\lambda<0$ such that $\Delta \varphi \leq \lambda <0$ in $\Omega\setminus U$. Consider the function $v: = u^g - \varphi$, then $v$ satisfies
    \begin{equation}
            \Delta v -\Delta \varphi \geq 0,\quad v\geq 0,\quad (\Delta v -\Delta \varphi) v = 0\quad\text{ in }\Omega\setminus U.
    \end{equation}
    Our desired result follows immediately by applying \cite[Chapter 2, Theorem 3.4]{Friedman82} to $v$.
\end{proof}

\begin{remark}
    As noted in \cite{C98,Weiss99,Blank01}, Caffarelli's dichotomy theorem is valid for $C^{2,\alpha}$ strictly superharmonic obstacles, that is, any point of the free boundary $\p H_g\setminus \p \Omega$ is either regular point or singular point, see e.g. detailed exposition in \cite[Chapter 5.6]{FR22}. 
    If $x\in \p H_g\setminus \p \Omega$ is a regular point, the free boundary is $C^{1,\alpha}$-smooth in a neighborhood of $x$. Consequently, we have 
    \begin{equation}
        \lim_{r\to 0}\frac{\abs{H_g\cap B(x,r)}}{\abs{B(x,r)}} =\frac{1}{2}.
    \end{equation}
    Conversely, if $x$ is a singular point of the free boundary, we have 
    \begin{equation}
        \lim_{r\to 0}\frac{\abs{H_g\cap B(x,r)}}{\abs{B(x,r)}} =1.
    \end{equation}
    Note that here $H_g$ denotes the non-contact set.
    Therefore, the lower bound $\delta$ in Lemma \ref{lm:free_db_density} can be improved to essentially the sharp constant $1/2$.
\end{remark}

\smallskip
Next, let us recall the following comparison principle.
\begin{proposition}[Comparison principle]
  \label{prop:ell_comparison}
  Let $g_1,g_2\in C^\infty(\p\Omega)$ satisfy
  \[
  g_1>g_2>\varphi\quad\text{on }\p\Omega.
  \]
  Then
  \[
  u^{g_1}\geq u^{g_2}\quad\text{in }\Omega.
  \]
  Consequently,
  \[
  K_{g_1}\subset K_{g_2},
  \qquad
  H_{g_2}\subset H_{g_1}.
  \]
\end{proposition}

\begin{proof}
Let $u_1:=u^{g_1}$ and $u_2:=u^{g_2}$, and suppose that
\[
S:=\{x\in\Omega:u_2(x)>u_1(x)\}
\]
is nonempty. Since $u_1\geq \varphi$, we have $u_2>\varphi$ on $S$, so $u_2$ is harmonic there. Therefore
\[
\Delta(u_1-u_2)=\Delta u_1\leq 0\quad\text{in }S,
\]
so $u_1-u_2$ is superharmonic in $S$. Since $g_1>g_2$ on $\p\Omega$, continuity of $u_1,u_2$ implies $S\Subset\Omega$, and on $\p S$ we have $u_1-u_2=0$. By the minimum principle for superharmonic functions,
\[
u_1-u_2\geq 0\quad\text{in }S,
\]
contradicting the definition of $S$. Hence $S=\varnothing$, so $u_1\geq u_2$ in $\Omega$. The set inclusions are immediate.
\end{proof}

\begin{lemma}[Strict comparison on the non-contact set]
  \label{lem:ell_strict_comparison}
  Under the hypotheses of Proposition \ref{prop:ell_comparison}, one has
  \[
  u^{g_1}>u^{g_2}\qquad\text{in }H_{g_1}.
  \]
\end{lemma}

\begin{proof}
Let $u_1:=u^{g_1}$ and $u_2:=u^{g_2}$. Fix $x\in H_{g_1}$. If $x\notin H_{g_2}$, then $u_2(x)=\varphi(x)<u_1(x)$ and there is nothing to prove.

If $x\in H_{g_2}$, let $C$ be the connected component of $H_{g_2}$ containing $x$. By Proposition \ref{prop:ell_comparison}, we have $H_{g_2}\subset H_{g_1}$, so both $u_1$ and $u_2$ are harmonic in $C$. Hence $w:=u_1-u_2$ is harmonic in $C$, and by Proposition \ref{prop:ell_comparison} it satisfies $w\ge 0$ there. By Lemma \ref{lem:ell_no_interior_components}, the component $C$ meets $\p\Omega$, and on $C\cap\p\Omega$ we have
\[
w=g_1-g_2>0.
\]
If $w(x)=0$ at some interior point of $C$, the strong minimum principle would force $w\equiv 0$ in $C$, contradicting the strict boundary values on $C\cap\p\Omega$. Therefore $w>0$ in $C$, and in particular $u_1(x)>u_2(x)$.
\end{proof}

\subsection{Continuity and left linearization}
Fix $g,h\in C^\infty(\p\Omega)$ with $g>\varphi|_{\p\Omega}$ and $h>0$ on $\p\Omega$. For $\varepsilon<0$ such that $g+\varepsilon h > \varphi|_{\p\Omega}$, we write
\begin{equation}
u_\varepsilon:=u^{g+\varepsilon h},
\qquad
u_0:=u^g.
\end{equation}
By Proposition \ref{prop:ell_comparison}, if $\varepsilon_2<\varepsilon_1$ then $u_{\varepsilon_2}\leq u_{\varepsilon_1}$.

\begin{proposition}[Continuity with respect to the boundary data]
  \label{prop:ell_continuity}
  As $\varepsilon\to 0$, one has
  \[
  u_\varepsilon\to u_0\quad\text{in }C^1(\overline\Omega).
  \]
\end{proposition}

\begin{proof}
Set $R_\varepsilon:=u_\varepsilon-u_0$. By \eqref{eq:ell_c1alpha_estimate}, the family $\{R_\varepsilon\}$ is bounded in $C^{1,\alpha}(\overline\Omega)$ for each fixed $\alpha\in(0,1)$. After passing to a subsequence, we may assume
\[
R_{\varepsilon_j}\to R_0\quad\text{in }C^1(\overline\Omega)
\]
as $\varepsilon_j\to 0$. Write $\widetilde u:=u_0+R_0$. Then $u_{\varepsilon_j}\to \widetilde u$ uniformly on $\overline\Omega$ and strongly in $H^1(\Omega)$, so $\widetilde u\geq \varphi$ and $\widetilde u|_{\p\Omega}=g$.

Let $G$ be the harmonic extension of $h$ to $\Omega$. Fix $v\in \mathcal K_\varphi^g$, and define
\[
v_{\varepsilon_j}:=\varphi+\bigl(v-\varphi+\varepsilon_j G\bigr)^+.
\]
If $\varepsilon_j>0$, then $v-\varphi+\varepsilon_j G\ge 0$ because $v\ge \varphi$ and $G\ge 0$, so in this case
\[
v_{\varepsilon_j}=v+\varepsilon_j G.
\]
If $\varepsilon_j<0$, the truncation may be active, and admissibility is preserved by the positive part. Since the map $w\mapsto w^+$ is Lipschitz on $H^1(\Omega)$; see \cite[Theorem~9.5]{Brezis}, we have $v_{\varepsilon_j}\to v$ in $H^1(\Omega)$. Moreover, $v_{\varepsilon_j}\geq \varphi$ and
\[
v_{\varepsilon_j}|_{\p\Omega}=g+\varepsilon_j h=g_{\varepsilon_j},
\]
so $v_{\varepsilon_j}\in \mathcal K_\varphi^{g_{\varepsilon_j}}$. Applying \eqref{def-variational-intro} to $u_{\varepsilon_j}$ gives
\[
\int_\Omega \nabla u_{\varepsilon_j}\cdot \nabla(v_{\varepsilon_j}-u_{\varepsilon_j})\,dx\geq 0.
\]
Passing to the limit yields
\[
\int_\Omega \nabla \widetilde u\cdot \nabla(v-\widetilde u)\,dx\geq 0,
\qquad \forall v\in \mathcal K_\varphi^g.
\]
Thus $\widetilde u$ solves the same variational inequality as $u_0$, and uniqueness gives $\widetilde u=u_0$. Since the sequence $\varepsilon_j\to 0$ was arbitrary, it follows that every sequence $\varepsilon_k\to 0$ has a subsequence $\varepsilon_{k_j}\to 0$ such that $R_{\varepsilon_{k_j}}\to 0$ in $C^1(\overline\Omega)$. In particular, $0$ is the only possible $C^1(\overline\Omega)$-cluster point of the family $\{R_\varepsilon\}$ as $\varepsilon\to 0$. Because \eqref{eq:ell_c1alpha_estimate} gives precompactness in $C^1(\overline\Omega)$, the whole family must converge to $0$ in $C^1(\overline\Omega)$.
\end{proof}

We now study the left difference quotient at the base solution $u^g$; the sets $H_\sigma$ and $K_\sigma$ below are the non-contact and contact sets associated with the solution $u_\sigma$.

\begin{lemma}
\label{lm:ell_H1_global_Reps}
    For $\varepsilon<0$, let
\[
R_\varepsilon:=u_\varepsilon-u_0,
\qquad
v_\varepsilon:=\frac{R_\varepsilon}{\varepsilon}.
\]
Then there exists a constant $C$ such that 
\begin{equation}
  \label{eq:ell_H1_global_Reps}
  \|R_\varepsilon\|_{H^1(\Omega)}\leq C|\varepsilon|,
  \qquad
  \|v_\varepsilon\|_{H^1(\Omega)}\leq C.
\end{equation}
\end{lemma}
\begin{proof}
    In view of \cite[Section 1.3.2]{PetrosyanShahgholianUraltseva2012}, for $\sigma\in\{\varepsilon,0\}$, we have $\Delta u_\sigma = 0$ in $H_{g+\sigma h}$, and $\Delta u_\sigma=\Delta\varphi$ a.e. on $K_{g+\sigma h}$. Therefore
\[
\Delta u_\sigma
=
\chi_{K_{g+\sigma h}}\Delta\varphi
=
(1-\chi_{H_{g+\sigma h}})\Delta\varphi
\qquad\text{a.e. in }\Omega.
\]
Subtracting the relations for $\sigma=\varepsilon$ and $\sigma=0$ gives
\begin{equation}
  \label{eq:ell_Reps_PDE}
  \begin{cases}
    \Delta R_\varepsilon
    =-\bigl(\chi_{\{u_\varepsilon>\varphi\}}-\chi_{\{u_0>\varphi\}}\bigr)\Delta \varphi
    \quad\text{in }\Omega,\\
    R_\varepsilon|_{\p\Omega}=\varepsilon h.
  \end{cases}
\end{equation}
Because $\varepsilon<0$, Proposition \ref{prop:ell_comparison} gives $u_\varepsilon\leq u_0$, hence
\[
R_\varepsilon\leq 0,
\qquad
\chi_{\{u_\varepsilon>\varphi\}}-\chi_{\{u_0>\varphi\}}\leq 0.
\]
Since $\Delta \varphi<0$, it follows from \eqref{eq:ell_Reps_PDE} that $\Delta R_\varepsilon\leq 0$, so $R_\varepsilon$ is superharmonic. By comparison with the harmonic function with boundary values $\varepsilon h$,
\begin{equation}
  \label{eq:ell_Reps_sup}
  \varepsilon\,\sup_{\p\Omega} h\leq R_\varepsilon\leq 0\quad\text{in }\Omega.
\end{equation}
The Caccioppoli estimate (see e.g. \cite[Lemma 3.27]{HKM}) yields, for every Lipschitz subdomain $\omega\Subset\Omega$,
\begin{equation}
  \label{eq:ell_H1_local_Reps}
  \|R_\varepsilon\|_{H^1(\omega)}\leq C_\omega |\varepsilon|.
\end{equation}
Since $K_g\Subset\Omega$, we may choose a smooth subdomain
\[
\Omega'\subset \Omega\setminus K_g,
\qquad \p\Omega\subset \p\Omega'.
\]
Since $\overline{\Omega'}\cap K_g=\emptyset$, we have $u_0>\varphi$ on $\overline{\Omega'}$. By Proposition \ref{prop:ell_continuity}, for all sufficiently small $|\varepsilon|$ we then also have $u_\varepsilon>\varphi$ on $\overline{\Omega'}$. Hence both $u_\varepsilon$ and $u_0$ solve $\Delta u=0$ on $\Omega'$, so $R_\varepsilon$ is harmonic there. Combining \eqref{eq:ell_H1_local_Reps} on an interior collar with the trace theorem on $\p\Omega'$ gives
\[
\|R_\varepsilon\|_{H^{1/2}(\p\Omega')}=O(|\varepsilon|),
\]
and hence, by the harmonic Dirichlet estimate on $\Omega'$,
\[
\|R_\varepsilon\|_{H^1(\Omega')}\leq C|\varepsilon|.
\]
Together with \eqref{eq:ell_H1_local_Reps}, this yields the global bound
\begin{equation}
  \label{eq:ell_H1_global_Reps}
  \|R_\varepsilon\|_{H^1(\Omega)}\leq C|\varepsilon|,
  \qquad
  \|v_\varepsilon\|_{H^1(\Omega)}\leq C.
\end{equation}
\end{proof}

\subsection{A rough Dirichlet problem}
At this point we freeze the base solution $u^g$ and its non-contact set $H_g$, and formulate the linearized problem in the natural space of $H^1$-functions vanishing on the contact set $K_g$.


We write
\[
H_0^1(H_g):=\overline{C_0^\infty(H_g)}^{\,H^1(\Omega)}.
\]
Let $E\subset \Omega$. The local Sobolev $2$-capacity of $E$ relative to $\Omega$ is defined by
\[
\operatorname{cap}_2(E,\Omega)
:=
\inf\left\{
\int_\Omega |\nabla \varphi|^2\,dx
:
\substack{
\varphi\in H_0^1(\Omega),\\
\varphi\ge 1 \text{ a.e.\ in a neighborhood of }E
}
\right\}.
\]
For any $w\in H^1(\Omega)$, by fine properties of Sobolev functions (see e.g. \cite[Theorem 6.2.1]{Adams95}), the limit
\begin{equation}
    \widetilde{w}(x) : = \lim_{r\to 0}\frac{\int_{B(x,r)}w(y)\, dy}{\abs{B(x,r)}}
\end{equation}
exists for each $x\in \Omega\setminus E$, where $E\subset \Omega$ is a Boreal set such that $\Cap_2(E,\Omega) = 0$. Since $\Cap_2(E,\RR^n) = 0$ implies that $E$ is a set of Lebesgue measure zero, $\widetilde{w}(x)$ is a quasi-continuous representative of $w$. For two functions $v,w$, we say that $v = w$ q.e. in a set $K\subset \RR^n$ if there exists a set $E\subset K$ such that $\Cap_2(E,K) = 0$ and $v=w$ in $K\setminus E$.
Since $H_g$ is open and $K_g = \Omega\setminus H_g$, in view of \cite[Theorem 9.1.3]{Adams95}, we have the following equivalent characterization
\begin{equation}
\label{eq:characterization}
    H_0^1(H_g)=\{w\in H_0^1(\Omega): \widetilde w=0 \text{ q.e. on }K_g\}.
\end{equation}
Since  $K_g$ is compactly contained in $\Omega$ we can choose $\chi_g\in C^\infty(\overline\Omega)$ such that
\[
0\leq \chi_g\leq 1,
\qquad
\chi_g=1 \text{ near }\p\Omega,
\qquad
\chi_g=0 \text{ in a neighborhood of }K_g.
\]
For $h\in C^\infty(\p\Omega)$, let $Fh\in C^\infty(\overline\Omega)$ be a smooth extension of $h$ in $\Omega$, and set
\[
E_g h:=\chi_g Fh,
\qquad
\mathcal V_g(h):=E_g h+H_0^1(H_g).
\]
Thus $\mathcal V_g(h)$ is the affine space encoding the outer boundary datum $h$ on $\p\Omega$, while the space $H_0^1(H_g)$ imposes quasi everywhere vanishing on the contact set $K_g$.
\begin{definition}[Rough Dirichlet problem on the non-contact set]
A function $v_{g,h}\in \mathcal V_g(h)$ is said to solve the rough Dirichlet problem on the non-contact set $H_g$ if
\begin{equation}
  \label{eq:ell_weak_mixed}
  \int_\Omega \nabla v_{g,h}\cdot \nabla \phi\,dx=0,
  \qquad \forall \phi\in H_0^1(H_g).
\end{equation}
\end{definition}
We note that above definition is independent of the particular choice of $\chi_g$ and $E_g h$, since any two such lifts differ by an element of $H_0^1(H_g)$.
\begin{theorem}
  \label{thm:ell_left_linearization}
  Let $g,h\in C^\infty(\p\Omega)$ with $g>\varphi|_{\p\Omega}$ and $h>0$ on $\p\Omega$. For $\varepsilon<0$, let
  \[
  u_\varepsilon:=u^{g+\varepsilon h},
  \qquad
  u_0:=u^g,
  \qquad
  v_\varepsilon:=\frac{u_\varepsilon-u_0}{\varepsilon}.
  \]
  Then there exists a unique function $v=v_{g,h}\in \mathcal V_g(h)$ solving the rough Dirichlet problem \eqref{eq:ell_weak_mixed} on the non-contact set $H_g$, and
  \[
  v_\varepsilon\rightharpoonup v_{g,h}
  \qquad\text{weakly in }H^1(\Omega)
  \]
  as $\varepsilon\uparrow 0$ through negative values. In particular, the left G\^ateaux derivative of the nonlinear Dirichlet-to-Neumann map exists in $H^{-1/2}(\p\Omega)$ and is given by
  \begin{equation}
    \label{eq:ell_left_derivative_DN}
    D^-\Lambda_\varphi(g)h
    :=\lim_{\varepsilon\uparrow 0}
    \frac{\Lambda_\varphi(g+\varepsilon h)-\Lambda_\varphi(g)}{\varepsilon}
    =\p_\nu v_{g,h}\big|_{\p\Omega}.
  \end{equation}
\end{theorem}

\begin{proof}
By Lemma \ref{lm:ell_H1_global_Reps}, the family $\{v_\varepsilon\}_{\varepsilon<0}$ is bounded in $H^1(\Omega)$, so after passing to a subsequence we may assume
\[
v_\varepsilon\rightharpoonup v\quad\text{weakly in }H^1(\Omega).
\]
Set
\[
w_\varepsilon:=v_\varepsilon-E_g h.
\]
Since $v_\varepsilon|_{\p\Omega}=h$ and $E_g h|_{\p\Omega}=h$, we have $w_\varepsilon\in H_0^1(\Omega)$. Moreover, $K_g\subset K_{g+\varepsilon h}$ for $\varepsilon<0$ by Proposition \ref{prop:ell_comparison}, so $u_\varepsilon=u_0=\varphi$ on $K_g$ and therefore $v_\varepsilon=0$ on $K_g$. Because $E_g h$ vanishes in a neighborhood of $K_g$, the function $w_\varepsilon$ is continuous and vanishes on $K_g$. Hence its quasi-continuous representative $\widetilde{w_\varepsilon} = 0$ q.e. on $K_g$, and the characterization \eqref{eq:characterization} gives
\[
w_\varepsilon\in H_0^1(H_g)
\qquad\text{for every }\varepsilon<0.
\]
Since $H_0^1(H_g)$ is a closed subspace of $H_0^1(\Omega)$, weak sequential closedness implies
\[
w:=v-E_g h\in H_0^1(H_g),
\]
so $v\in \mathcal V_g(h)$.

Next let $\phi\in C_c^{\infty}(H_g)$. Because $\supp \phi\Subset H_g$ and $u_0>\varphi$ on $H_g$, Proposition \ref{prop:ell_continuity} implies that for all sufficiently small negative $\varepsilon$ we also have $u_\varepsilon>\varphi$ on $\supp \phi$. Thus both $u_\varepsilon$ and $u_0$ are harmonic on $\supp \phi$, and so
\[
\int_\Omega \nabla v_\varepsilon\cdot \nabla \phi\,dx=0.
\]
Passing to the limit gives
\[
\int_\Omega \nabla v\cdot \nabla \phi\,dx=0
\qquad\forall \phi\in C_c^{\infty}(H_g).
\]
Since $C_c^{\infty}(H_g)$ is dense in $H_0^1(H_g)$ by definition, we obtain \eqref{eq:ell_weak_mixed}. Thus $v$ solves the rough Dirichlet problem on the non-contact set $H_g$.

To prove uniqueness, let $v_1,v_2\in \mathcal V_g(h)$ solve \eqref{eq:ell_weak_mixed}. Their difference $w:=v_1-v_2$ lies in $H_0^1(H_g)$ and satisfies
\[
\int_\Omega \nabla w\cdot \nabla \phi\,dx=0,
\qquad \forall \phi\in H_0^1(H_g).
\]
Taking $\phi=w$ gives
\[
\int_\Omega |\nabla w|^2\,dx=0,
\]
so $w=0$ in $H^1(\Omega)$. Therefore the solution is unique, and the whole family $v_\varepsilon$ converges weakly to $v_{g,h}$.

Finally, choose a smooth collar subdomain $U\subset \Omega$ such that $\p\Omega\subset \p U$ and $\overline U\subset H_g\cup \p\Omega$. By Proposition \ref{prop:ell_continuity}, after shrinking $|\varepsilon|$ if necessary we also have $\overline U\subset H_{g+\varepsilon h}\cup \p\Omega$, so $v_\varepsilon$ is harmonic in $U$ for all such $\varepsilon$. Let $\Sigma:=\p U\setminus \p\Omega$. For any $\eta\in H^{1/2}(\p\Omega)$, choose $\Phi\in H^1(U)$ such that
\[
\Phi|_{\p\Omega}=\eta,
\qquad
\Phi|_{\Sigma}=0.
\]
Since $v_\varepsilon$ is harmonic in $U$, Green's identity gives
\[
\left\langle \frac{\Lambda_\varphi(g+\varepsilon h)-\Lambda_\varphi(g)}{\varepsilon},\eta\right\rangle_{\p\Omega}
=\int_U \nabla v_\varepsilon\cdot \nabla \Phi\,dx.
\]
Passing to the limit yields
\[
\lim_{\varepsilon\uparrow 0}
\left\langle \frac{\Lambda_\varphi(g+\varepsilon h)-\Lambda_\varphi(g)}{\varepsilon},\eta\right\rangle_{\p\Omega}
=\int_U \nabla v\cdot \nabla \Phi\,dx.
\]
The right-hand side is the weak normal trace of $v$ on $\p\Omega$; since $v$ is harmonic in $U$, this trace agrees with the classical normal derivative. This proves \eqref{eq:ell_left_derivative_DN}.
\end{proof}

\begin{proposition}
  \label{prop:ell_mixed_positive}
  Let $g,h\in C^\infty(\p\Omega)$ with $g>\varphi|_{\p\Omega}$ and $h>0$ on $\p\Omega$, and let $v=v_{g,h}$ be the solution given by Theorem \ref{thm:ell_left_linearization}. Then $v\in C^\infty(H_g)$ and
  \[
  v>0\qquad\text{in }H_g.
  \]
\end{proposition}

\begin{proof}
Let $U\Subset U'\Subset H_g$. For all sufficiently small negative $\varepsilon$, Proposition \ref{prop:ell_continuity} gives $U'\Subset H_{g+\varepsilon h}$, so $v_\varepsilon$ is harmonic in $U'$. Standard interior estimates for harmonic functions (see e.g. \cite[Theorem 17.1.3, Lemma 17.1.5]{Hormander3}) imply that $\{v_\varepsilon\}$ is bounded in $C^m(U)$ for every $m\ge 0$, and therefore precompact in $C^\infty(U)$. Since every subsequential limit must coincide with the weak $H^1$-limit $v$, we conclude that $v_\varepsilon\to v$ in $C^\infty(U)$. Thus $v\in C^\infty(H_g)$.


By Lemma \ref{lem:ell_strict_comparison}, for every sufficiently small negative $\varepsilon$ one has
\[u^g>u^{g+\varepsilon h}\qquad\text{in }H_g.
\]
Since $\varepsilon<0$, this implies
\[
v_\varepsilon=\frac{u^{g+\varepsilon h}-u^g}{\varepsilon}>0\qquad\text{in }H_g.
\]
Passing to the limit on compact subsets of $H_g$ gives $v\ge 0$ in $H_g$. Let $A$ be a connected component of $H_g$. By Lemma \ref{lem:ell_no_interior_components}, $A$ meets $\p\Omega$. Since $v$ is harmonic in a collar neighborhood of $\p\Omega$ and $v|_{\p\Omega}=h>0$, continuity implies that $v>0$ in a nonempty open subset of $A$. The strong maximum principle then yields $v>0$ throughout $A$. Since $A$ was arbitrary, $v>0$ in $H_g$.
\end{proof}

\section{Inverse Elliptic Obstacle Problem} \label{sec-elliptic}
\label{sec:inverse_elliptic}

This section is devoted to proving Theorem \ref{thm-main-1}. To this purpose, we first reduce our inverse problem to the case where the trace of obstacle vanishes on the boundary. Let $\varphi_1, \varphi_2\in C^{2,1}(\overline{\Omega})$ with $\varphi_1|_{\p\Omega} = \varphi_2|_{\p\Omega}$. Let $\psi=\psi(x)$ be the solution to
\begin{equation}
  \label{eq:elliptic}
  \begin{cases}
    \Delta \psi=0,&\text{in }\Omega,\\
    \psi|_{\p\Omega}=\varphi_1|_{\p\Omega}.
  \end{cases}
\end{equation}
For $i=1,2$, set
\[
\widetilde{\varphi}_i=\varphi_i-\psi,\qquad
\widetilde{f}=f-\psi|_{\p\Omega},\qquad
u_{i}^{\widetilde{f}}=u_{\varphi_i}^f-\psi,
\]
where $u_{\varphi_i}^f$ solves \eqref{def-obstacle-intro} associated with the obstacle $\varphi_i$ and the boundary value $f$.
Then $u_{i}^{\widetilde{f}}$ satisfies \eqref{def-obstacle-intro} with the normalized obstacle $\widetilde{\varphi}_i$ and Dirichlet condition $\widetilde{f}$.
Consequently, for both $i= 1,2$ and every $\widetilde{f}\in \cA_{\tphi_i}$ there holds
\begin{equation}
    \Theta_{\tphi_i}(\widetilde{f}) = \p_\nu u_i^{\widetilde{f}} = \p_\nu u_{\varphi_i}^f - \p_\nu \psi = \Theta_{\varphi_i}(f) - \p_\nu \psi.
\end{equation}
Hence $\Theta_{\varphi_1} = \Theta_{\varphi_2}$ implies that $\Theta_{\tphi_1} = \Theta_{\tphi_2}$.
We notice that provided $\tphi_1 = \tphi_2$, shifting back by $\psi$ yields
\begin{equation}
    \varphi_1 = \tphi_1 +\psi = \tphi_2 +\psi = \varphi_2.
\end{equation}

According to above discussion, Theorem \ref{thm-main-1} follows from the following slightly stronger form.

\begin{theorem}
  \label{thm:ellptic_obstacle}
  Let $\tphi_1,\tphi_2\in C^{2,1}(\overline\Omega)$ satisfy
  \[
  \tphi_1|_{\p\Omega}=\tphi_2|_{\p\Omega}=0,
  \qquad
  \Delta \tphi_1<0,\ \Delta \tphi_2<0 \quad\text{in }\Omega.
  \]
  Suppose
  \[
  \Theta_{\tphi_1}(g)=\Theta_{\tphi_2}(g),
  \]
  for every $g\in C^\infty(\p\Omega)$ with $g>0$ on $\p\Omega$. Then
  \[
  \tphi_1=\tphi_2 \quad\text{in }\Omega.
  \]
\end{theorem}

\begin{remark}
In fact, in the proof of Theorem \ref{thm:ellptic_obstacle}, it suffices to assume that the Dirichlet-to-Neumann maps $\Theta_{\tphi}$ for the normalized obstacles $\tphi$ coincide on the ray $\{t g \mid t> 0\}$ for some $g\in C^\infty(\partial \Omega)$ with $g > 0$. 
In the form of the setting of Theorem \ref{thm-main-1}, this is equivalent to the coincidence of the Dirichlet-to-Neumann maps $\Theta_{\varphi}$ for the obstacles $\varphi$ on the ray $\{\varphi|_{\p\Omega} + t g \mid t > 0\}$ for some $g\in C^\infty(\partial \Omega)$ with $g > 0$.
\end{remark}

\subsection{Determining the non-contact set}
Theorem \ref{thm:ell_left_linearization} identifies the left linearization of the nonlinear Dirichlet-to-Neumann map with the Dirichlet-to-Neumann map of a rough Dirichlet problem on the \emph{unknown non-contact set}. 
The following theorem shows that this dependence is in fact injective at the same capacity level: equality of the linearized Cauchy data forces equality of the non-contact sets up to sets of zero Sobolev $2$-capacity. This is the natural geometric uniqueness statement for the linearized problem. To recover the obstacle itself, we vary scalar multiples of a fixed positive boundary profile and integrate the resulting family of solutions of the rough Dirichlet problem on the non-contact set in the amplitude parameter. The details of the integration step are given in Section \ref{sec:integration}.

Let us denote the difference between two sets $S_1$ and $S_2$ by $S_1\triangle S_1$, i.e. 
$$S_1\triangle S_2 : = (S_1\setminus S_2) \cup (S_2\setminus S_1).$$ 
\begin{theorem}
  \label{thm:ell_linearized_set_uniqueness}
  Let $\psi_1,\psi_2\in C^{2,1}(\overline\Omega)$ satisfy
  \[
  \psi_i|_{\p\Omega}=0,
  \qquad
  \Delta \psi_i<0\quad\text{in }\Omega,
  \qquad i=1,2.
  \]
  Fix $g,h\in C^\infty(\p\Omega)$ with $g>0$ and $h>0$ on $\p\Omega$. For $i=1,2$, let $u_i^g$ be the solution to \eqref{def-obstacle-intro} associated with obstacle $\psi_i$ and boundary value $g$, and set
  \[
  H_i:=\{x\in \Omega \mid u_i^g(x)>\psi_i(x)\},
  \qquad
  K_i:=\{x\in \Omega \mid u_i^g(x)=\psi_i(x)\}.
  \]
  Suppose the corresponding left linearized Dirichlet-to-Neumann maps agree in the direction $h$:
  \[
  D^-\Lambda_{\psi_1}(g)h=D^-\Lambda_{\psi_2}(g)h
  \qquad\text{in }H^{-1/2}(\p\Omega).
  \]
  Then
  \[
  \operatorname{cap}_2(H_1\triangle H_2,\Omega)=0.
  \]
  Equivalently, $K_1$ and $K_2$ agree up to a set of Sobolev $2$-capacity zero.
\end{theorem}


The following lemma is well known and we include a brief proof here for the convenience of readers.

\begin{lemma}
  \label{lem:ell_cap_zero_no_disconnect}
  Let $B\Subset\Omega$ be a connected open set and let $F\subset B$ satisfy
  \[
  \operatorname{cap}_2(F,B)=0.
  \]
  Then $B\setminus F$ is connected.
\end{lemma}




\begin{proof}
    According to \cite[Theorem 4.17]{Gariepy25}, there holds $\cH^{s}(F) = 0$ for any $s>n-2$. Hence we have $\dim_H (F)\leq n-2$, where $\dim_H$ stands for the Hausdorff dimension (see e.g. \cite[Definition 2.2]{Gariepy25}). It is known due to \cite[Section 3.1]{Edgar98} that $\text{ind}(F)\leq \dim_H (F)$, where $\text{ind}(F)$ is the small inductive dimension (also called Menger-Urysohn dimension) of $F$ (see e.g. \cite[Definition 1.1.1]{engelking1995} for definition). 
    By Mazurkiewicz's theorem, e.g. \cite[Theorem 1.8.18]{engelking1995}, we conclude that $B\setminus F$ is connected.
\end{proof}

\begin{lemma}
    \label{lm:bd_connected_component}
    Let $U$ be an open set and $K$ a closed set in $\RR^n$. Let $V$ be a connected component of $U\setminus K$, then we have $\p V\cap U \subset \p K$. 
\end{lemma}
\begin{proof}
    Since $\p V\subset \overline{V}\subset \overline{U\setminus K} \subset \overline{U}\setminus K^\circ$.
    It remains to show that $\p V\cap U\subset K$. To get a contradiction, suppose that there exists a point $x\in \p V\cap(U\setminus K)$. Since $U\setminus K$ is open, we can find an open ball $B(x,r)$ such that $B(x,r)\subset U\setminus K$. Let $\widetilde{V}$ be the connected component in $U\setminus K$ such that $B(x,r)\subset \widetilde{V}$. Since $x\not\in V$, then $V\neq \widetilde{V}$. And thus $V\cap \widetilde{V} = \emptyset$ since they are distinct connected components. Therefore, $B(x,r)\cap V = \emptyset$, and this leads to a contradiction to $x\in \p V$.
\end{proof}

\begin{proof}[Proof of Theorem \ref{thm:ell_linearized_set_uniqueness}]
For $i=1,2$, let $v_i:=v_{g,h}^{(i)}$ denote the solution of the rough Dirichlet problem on the non-contact set given by Theorem \ref{thm:ell_left_linearization} for the obstacle $\psi_i$. In other words, $v_i$ is the linearized solution obtained by perturbing the boundary value $g$ in the direction $h$ for the obstacle $\psi_i$; equivalently, $v_i$ solves the rough Dirichlet problem on $H_i$ with outer boundary datum $h$.
Without loss of generality, we may choose $E_g h$ such that $E_g h = 0$ on $K_1\cup K_2$ and $E_g h|_{\p \Omega} = h$.

Then $v_i\in E_g h+H_0^1(H_i)$, $v_i$ is harmonic in $H_i$, and
\[
\p_\nu v_i\big|_{\p\Omega}=D^-\Lambda_{\psi_i}(g)h.
\]
By hypothesis, $v_1$ and $v_2$ have the same Dirichlet and Neumann data on $\p\Omega$.

Since $g>0$ and $\psi_i|_{\p\Omega}=0$, both $H_1$ and $H_2$ contain collar neighborhoods of $\p\Omega$. Hence there exists a connected collar domain $U_0\subset H_1\cap H_2$ with $\p\Omega\subset\p U_0$. The function $v_1-v_2$ is harmonic in $U_0$ and has vanishing Cauchy data on $\p\Omega$, so standard boundary unique continuation for harmonic functions yields
\begin{equation}
  \label{eq:linearized_common_collar}
  v_1=v_2\qquad\text{in }U_0.
\end{equation}

Let us show that
\begin{equation}
  \label{eq:cap_K2_H1_zero}
  \operatorname{cap}_2(K_2\cap H_1,\Omega)=0.
\end{equation}
Suppose, on the contrary, that $\Cap_2(K_2\cap H_1,\Omega) > 0$. Since $\RR^n$ is second-countable, we may cover $H_1$ by a family $\{B_{r_j}(x_j)\}_{j=1}^\infty$ of countably many balls satisfying $B_{r_j}(x_j)\Subset H_1$ for each $j>0$. By countable subadditivity of relative capacity,
\[
0<\operatorname{cap}_2(K_2\cap H_1,\Omega)
\le \sum_{j=1}^\infty \operatorname{cap}_2((K_2\cap H_1)\cap B_{r_j}(x_j),\Omega).
\]
Hence for some $n>0$ we have
\[
\operatorname{cap}_2((K_2\cap H_1)\cap B_{r_m}(x_m),\Omega)>0.
\]
Monotonicity of relative capacity with respect to the ambient domain gives
\[
\operatorname{cap}_2(B_{r_m}(x_m)\cap K_2,\Omega)
\le
\operatorname{cap}_2(B_{r_m}(x_m)\cap K_2,B_{r_m}(x_m)),
\]
and therefore
\[
\operatorname{cap}_2(B_{r_m}(x_m)\cap K_2,B_{r_m}(x_m))>0.
\]
Due to Lemma \ref{lem:ell_no_interior_components}, we choose a continuous curve $\gamma\subset H_1$ joining a point $y\in U_0$ to $x$. Cover $\gamma$ by finitely many overlapping balls
\[
B_1,\dots,B_N
\]
with the following properties:
\begin{enumerate}
  \item $B_1\Subset U_0$;
  \item $B_k\Subset H_1$ for $k=2,\dots,N$;
  \item $B_k\cap B_{k+1}\neq\varnothing$ for $k=1,\dots,N-1$;
  \item $B_N=B_{r_m}(x_m)$.
\end{enumerate}

Let $M\in \{2,\cdots,N\}$ be the first index such that
\[
\operatorname{cap}_2(K_2 \cap B_M,B_M)>0.
\]
By minimality, for each $k<M$ we have $\operatorname{cap}_2(B_k\cap K_2,B_k)=0$. For any $\phi\in C_0^{\infty}(B_k)$,
since $\operatorname{cap}_2(K_2\cap B_k,B_k)=0$, we can conclude that the quasi-continuous representative $\widetilde\phi=0$ q.e. on $K_2$.
Therefore, $\phi\in H_0^1(H_2)$ by virtue of \eqref{eq:characterization}. Then by definition, we have
\begin{equation}
    \int_{B_k}\phi \Delta v_2 dx = \int_{B_k} \nabla v_2 \nabla \phi = 0.
\end{equation}
Since $\phi\in C_0^\infty(B_k)$ is arbitrary,
it follows that $v_2$ is harmonic in $B_k$.
The function $v_1$ is also harmonic there because $B_k\Subset H_1$. Starting from \eqref{eq:linearized_common_collar} in $B_1$, unique continuation across the nonempty overlaps $B_k\cap B_{k+1}$ gives
\begin{equation}
  \label{eq:linearized_equal_before_M}
  v_1=v_2\qquad\text{in }\bigcup_{k=1}^{M-1} B_k.
\end{equation}

Set
$U:=B_M\cap H_2$,
and let $C$ be the connected component of $U$ that meets $B_{M-1}\cap B_M$. By \eqref{eq:linearized_equal_before_M}, the functions $v_1$ and $v_2$ agree on a nonempty open subset of $C$. Since both are harmonic in $C$, unique continuation implies
\begin{equation}
  \label{eq:linearized_equal_in_C}
  v_1=v_2\qquad\text{in }C.
\end{equation}

By Proposition \ref{prop:ell_mixed_positive}, the function $v_1$ is continuous and strictly positive in $B_M\Subset H_1$. Thus there exists a constant $c_0$ such that
\begin{equation}
  \label{eq:v1_positive_lower_bound}
  v_1\ge c_0>0\qquad\text{on }\overline{B_M}.
\end{equation}
Recall that $v_2 = E_g h + w$ for some $w\in H^1_0(H_2)$ and thus $\widetilde{v_2} = \widetilde{E_g h} + \widetilde{w}$. Since $E_g h = 0$ in a neighborhood of $K_2$, we have $\widetilde{E_g h} = 0$ on $K_2$. By definition, there holds $\widetilde{w} = 0$ q.e. on $K_2$.
Therefore the quasi-continuous representative $\widetilde v_2$ satisfies
\begin{equation}
  \label{eq:v2_zero_qe_on_K2}
  \widetilde v_2=0\qquad\text{q.e. on }K_2\cap B_M.
\end{equation}
Moreover, on $C$ the function $v_2$ is smooth and coincides with $v_1$ by \eqref{eq:linearized_equal_in_C}. Hence there holds
\begin{equation}
    v_2\geq c_0 >0 \text{ on }C \cap B_M.
\end{equation}

Next, we separate our discussion into two cases:

(\romannumeral1) $B_M \setminus \overline{C}\neq \emptyset$.
Since $\p C\cap B_M \subset \p H_2\setminus \p\Omega$ is a part of the free boundary, $\overline{C}$ is of finite perimeter in $B_M$ according to \cite[Corollary 3]{caffarelli1981remark}. Thus we may consider the reduced boundary $\p^\ast C\cap B_M$. Let us show that $\Cap_2(\p^\ast C\cap B_M, B_M) > 0$. Since $B_M\setminus \overline{C}$ is open, then $\abs{B_M\setminus \overline{C}}>0$. According to \cite[Theorem 5.4.3]{Ziemer89}, there holds 
\begin{equation}
    0<\min\left(\abs{\overline{C}\cap B_M}, \abs{B_M\setminus \overline{C}}\right)^{\frac{n-1}{n}}\lesssim P(\overline{C}; B_M).
\end{equation}
Here $P(\overline{C}; B_M)$ stands for the total variation of $\chi_{\overline{C}}$ in $B_M$, see e.g. \cite[Definition 5.4.1]{Ziemer89}.
In the view of De Giorgi’s structure theorem (see e.g. \cite[Theorem 15.9]{Maggi12}), we have $\cH^{n-1}(\p^\ast C\cap B_M)>0$. Notice that this implies that $\Cap_2(\p^\ast C\cap B_M, B_M)>0$ by \cite[Theorem 2.6.16]{Ziemer89}.
Next, we show that $\widetilde{v_2} >0$ on q.e. $\p^\ast C \cap B_M$.
Let us consider the set of points of density $t$ of $C$,
\begin{equation}
    C^{(t)}=\left\{x\in \RR^n\mid \lim_{\varepsilon\to0^+}\frac{\abs{C\cap B(x,\varepsilon)}}{\abs{B(x,\varepsilon)}}=t \right\}.
\end{equation}
According to Federer's theorem (see e.g. \cite[Theorem 16.2]{Maggi12}), we have $\p^\ast C\cap B_M \subset C^{(1/2)}\cap B_M$. Therefore, for q.e. $x\in \p^\ast C\cap B_M$, there holds
\begin{equation}
\label{eq:nonzero_wv_1}
    \wv_2(x) = \lim_{r\to 0}\frac{\int_{B(x,r)}v_2(y)\ dy }{\abs{B(x,r)}}\geq \lim_{r\to 0}\frac{\int_{B(x,r)\cap C} v_2(y)\ dy }{\abs{B(x,r)}}\geq \lim_{r\to 0}\frac{c\abs{B(x,r)\cap C}}{\abs{B(x,r)}}  = \frac{c}{2}.
\end{equation}
Finally, in view of Lemma \ref{lm:bd_connected_component}, $\p C\cap B_M\subset \p K_2$ and thus $\p^\ast C\cap B_M\subset \p C\cap B_M\subset K_2\cap B_M$. Hence \eqref{eq:nonzero_wv_1} leads a contradiction to $\wv_2 = 0$ q.e. on $K_2\cap B_M$ and \eqref{eq:cap_K2_H1_zero} is proved in this case.

(\romannumeral2)$B_M\setminus \overline{C} = \emptyset$. 
We claim that
\begin{equation}
  \label{eq:positive_capacity_F}
  \operatorname{cap}_2(\p C\cap B_M,B_M)>0.
\end{equation}
Indeed, to get a contradiction, assume that $\operatorname{cap}_2(\p C\cap B_M,B_M)=0$. By Lemma \ref{lem:ell_cap_zero_no_disconnect}, the set $B_M\setminus \p C$ is connected. On the other hand, if $G:=B_M\setminus\overline C$ were nonempty, then $C$ and $G$ would be disjoint nonempty open subsets of $B_M\setminus \p C$, contradicting connectedness. Thus $G=\varnothing$, hence $C=B_M \setminus \p C$. 
It follows that
\[
K_2\cap B_M\subset B_M\setminus C=\p C\cap B_M,
\]
which contradicts the choice of $B_M$. This proves \eqref{eq:positive_capacity_F}.
Recall that $C$ is a connected component of $B_M \cap H_2$ by definition, then we have $C = H_2\cap B_M$ and thus $\p C = \p H_2\cap B_M$.
d
According to Lemma \ref{lm:free_db_density}, for any $x\in \p C\cap B_M \subset \p H_2$, there exists a constant $\delta>0$ such that
\begin{equation}
    \liminf_{r\to 0}\frac{\abs{C\cap B(x,r)}}{\abs{B(x,r)}} = \liminf_{r\to 0}\frac{\abs{H_2\cap B(x,r)}}{\abs{B(x,r)}}\geq \delta.
\end{equation}
Thus for q.e. $x\in \p C\cap B_M$, there holds
\begin{align}
    \wv_2(x) &= \lim_{r\to 0}\frac{\int_{B(x,r)}v_2(y)\ dy }{\abs{B(x,r)}}\geq \liminf_{r\to 0}\frac{\int_{B(x,r)\cap C} v_2(y)\ dy }{\abs{B(x,r)}}\\
    &\geq \liminf_{r\to 0}\frac{c_0\abs{B(x,r)\cap C}}{\abs{B(x,r)}}  \geq \delta c_0 >0.
\end{align}
However, since $\Cap_2(\p C\cap B_M, B_M)>0$,
this leads a contradiction to $\wv = 0$ q.e. on $K_2\cap B_M$. Consequently, \eqref{eq:cap_K2_H1_zero} is proved in this case.

Interchanging the roles of $(1,2)$ and $(2,1)$ yields
\[
\operatorname{cap}_2(K_1\cap H_2,\Omega)=0.
\]
Since
\[
H_1\triangle H_2=(H_1\cap K_2)\cup(H_2\cap K_1),
\]
subadditivity of capacity shows that $\operatorname{cap}_2(H_1\triangle H_2,\Omega)=0$, as claimed.
\end{proof}

\subsection{Determining the obstacle}
\label{sec:integration}
We now show that the geometric uniqueness result of Theorem \ref{thm:ell_linearized_set_uniqueness} leads to pointwise uniqueness of the obstacle.

Fix once and for all a boundary function $f\in C^\infty(\p\Omega)$ with $f>0$ on $\p\Omega$. For a normalized obstacle $\psi$, let
\[
U_\psi(\lambda):=u^{\lambda f}_\psi,
\qquad \lambda\ge 0,
\]
where $u^{\lambda f}_\psi$ solves \eqref{def-obstacle-intro} with Dirichlet boundary condition $\lambda f$ and obstacle $\psi$,
and for $\lambda>0$ write
\[
H_{\psi,\lambda}:=\{U_\psi(\lambda)>\psi\},
\qquad
K_{\psi,\lambda}:=\Omega\setminus H_{\psi,\lambda}.
\]
We also denote by
\[
V_\psi(\lambda):=v_{\lambda f,f}
\]
the solution of the rough Dirichlet problem on the non-contact set given by Theorem \ref{thm:ell_left_linearization}. Thus $U_\psi(\lambda)$ is the nonlinear obstacle solution for boundary value $\lambda f$, $H_{\psi,\lambda}$ is its non-contact set, and $V_\psi(\lambda)$ is the corresponding left linearized solution in the direction $f$.

\begin{proposition}
  \label{cor:ell_local_lipschitz_lambda}
  For every normalized obstacle $\psi$, the map
  \[
  (0,\infty)\ni \lambda\mapsto U_\psi(\lambda)\in H^1(\Omega)
  \]
  is locally Lipschitz.
\end{proposition}
\begin{proof}
   Let $\lambda\in (0,\infty)$ be arbitrary fixed and $\delta\in \RR$ be such that $\lambda-\delta >0$. For any $z_1, z_2\in (\lambda-\delta,\lambda+\delta)$ with $z_1 < z_2$, analogous to the proof for Lemma \ref{lm:ell_H1_global_Reps} (replace $R_\varepsilon$ by $U_\psi(z_1)-U_\psi(z_2)$), for any smooth subdomain $\omega\Subset \Omega$, there holds
   \begin{equation}
   \label{eq:prop_locol_lipschitz}
       \norm{U_\psi(z_1)-U_\psi(z_2)}_{H^1(\omega)}\leq C_\omega(z_2-z_1).
   \end{equation}
   According to Proposition \ref{prop:ell_comparison}, we have 
   \begin{equation}
       K_{\psi,z_2}\subset K_{\psi,z_1}\subset K_{\psi,\lambda-\delta}.
   \end{equation}
   Then we can find a smooth subdomain $\widetilde{\omega}\Subset \Omega$ such that $K_{\psi,\lambda-\delta}\subset \widetilde{\omega}$. Combining \eqref{eq:prop_locol_lipschitz} with the trace theorem on $\widetilde{\omega}$ gives 
   \begin{equation}
       \norm{U_\psi(z_1)-U_\psi(z_2)}_{H^{1/2}(\p \widetilde{\omega})}\leq C_{\widetilde{\omega}}(z_2-z_1),
   \end{equation}
   Notice that 
   \begin{equation}
   \begin{cases}
       \Delta(U_\psi(z_2)-U_{\psi}(z_1)) = 0 \text{ in }\Omega\setminus \widetilde{\omega},\\
       (U_\psi(z_2)-U_{\psi}(z_1))|_{\p\Omega} = (z_2 - z_1)f.
   \end{cases}
   \end{equation}
   Then the regularity estimate for $U_\psi(z_2)-U_{\psi}(z_1)$ on the collar domain $\Omega\setminus \widetilde{\omega}$ reads
   \begin{equation}
   \label{eq:prop_local_lipschitz_1}
       \norm{U_\psi(z_2)-U_{\psi}(z_1)}_{H^1(\Omega\setminus \widetilde{\omega})} \leq C_{\widetilde{\omega}, f} (z_2-z_1),
   \end{equation}
  Combining \eqref{eq:prop_locol_lipschitz} and \eqref{eq:prop_local_lipschitz_1}, we have 
   \begin{equation}
       \norm{U_\psi(z_2)-U_{\psi}(z_1)}_{H^1(\Omega)}\leq C (z_2-z_1),
   \end{equation}
    where $C$ is a constant that only depends on $\widetilde{\omega}$ and $f$. Since $z_1$ and $z_2$ are chosen arbitrarily, the map $U_\psi(\lambda)$ is Lipschitz in $(\lambda-\delta,\lambda+\delta)$ and thus locally Lipschitz.
\end{proof}

\begin{lemma}
  \label{lem:ell_equal_linearized_velocities}
  Let $\psi_1,\psi_2\in C^{2,1}(\overline\Omega)$ satisfy
  \[
  \psi_1|_{\p\Omega}=\psi_2|_{\p\Omega}=0,
  \qquad
  \Delta\psi_1<0,\ \Delta\psi_2<0 \quad\text{in }\Omega,
  \]
  and suppose that
  \[
  \Theta_{\psi_1}(g)=\Theta_{\psi_2}(g)
  \]
  for every $g\in C^\infty(\p\Omega)$ with $g>0$ on $\p\Omega$. Then for every $\lambda>0$,
  \[
  V_{\psi_1}(\lambda)=V_{\psi_2}(\lambda)
  \qquad\text{in } H^1(\Omega).
  \]
\end{lemma}

\begin{proof}
Fix $\lambda>0$. For each $i$, the function $V_{\psi_i}(\lambda)$ is the solution of the rough Dirichlet problem on the set $H_{\psi_i,\lambda}$ associated with the nonlinear state $U_{\psi_i}(\lambda)$. Equality of the nonlinear Dirichlet-to-Neumann maps implies equality of their left derivatives at $\lambda f$ in the direction $f$. By Theorem \ref{thm:ell_linearized_set_uniqueness},
\[
\operatorname{cap}_2(H_{\psi_1,\lambda}\triangle H_{\psi_2,\lambda},\Omega)=0.
\]
Hence
\[
H_0^1(H_{\psi_1,\lambda})=H_0^1(H_{\psi_2,\lambda}).
\]
Since $\lambda f>0$ on $\p\Omega$ and both obstacles vanish there, the sets $H_{\psi_1,\lambda}$ and $H_{\psi_2,\lambda}$ contain collar neighborhoods of $\p\Omega$. Choose $\chi\in C^\infty(\overline\Omega)$ such that
\[
0\le \chi\le 1,
\qquad
\chi=1 \text{ near } \p\Omega,
\qquad
\supp \chi
\subset H_{\psi_1,\lambda}\cap H_{\psi_2,\lambda}.
\]
Let $Ef\in C^\infty(\overline\Omega)$ be any smooth extension of $f$ and set $L:=\chi Ef$. Then both $V_{\psi_1}(\lambda)$ and $V_{\psi_2}(\lambda)$ belong to the same affine space
\[
L+H_0^1(H_{\psi_1,\lambda})=L+H_0^1(H_{\psi_2,\lambda}),
\]
and both satisfy the same weak equation
\[
\int_\Omega \nabla v\cdot\nabla \phi\,dx=0
\qquad \forall \phi\in H_0^1(H_{\psi_1,\lambda})=H_0^1(H_{\psi_2,\lambda}).
\]
By uniqueness of the rough Dirichlet problem on the non-contact set, the two solutions coincide.
\end{proof}

\begin{lemma}
  \label{lem:ell_banach_constancy}
  Let $X$ be a Banach space and let $W:(0,\infty)\to X$ be locally Lipschitz. Assume that for every $\lambda>0$ the left derivative $D^-W(\lambda)$ exists in the weak sense and satisfies
  \[
  D^-W(\lambda)=0.
  \]
  Then $W$ is constant on $(0,\infty)$.
\end{lemma}

\begin{proof}
Fix $\ell\in X^*$ and define
\[
w_\ell(\lambda):=\ell(W(\lambda)).
\]
Then $w_\ell$ is a locally Lipschitz scalar function on $(0,\infty)$, so it is absolutely continuous on compact intervals and differentiable almost everywhere.
At every point where $w_\ell$ is differentiable, its classical derivative agrees with its left derivative. That is,
\begin{equation}
    w_\ell^\prime(\lambda) = D^-w_\ell(\lambda)=\ell(D^-W(\lambda))=0, \text{ a.e. }\lambda\in (0,\infty).
\end{equation}
Hence $w_\ell$ is constant on $(0,\infty)$.
Since this holds for every $\ell\in X^*$ and the dual separates points, $W$ itself is constant.
\end{proof}


\begin{proof}[Proof of Theorem \ref{thm:ellptic_obstacle}]

For $i=1,2$, define
\[
U_i(\lambda):=U_{\tphi_i}(\lambda)=u^{\lambda f}_{\tphi_i},
\qquad
V_i(\lambda):=V_{\tphi_i}(\lambda),
\qquad
\lambda>0.
\]
Here $U_i(\lambda)$ denotes the nonlinear obstacle solution for the obstacle $\tphi_i$ and boundary value $\lambda f$, while $V_i(\lambda)$ is its left linearization with respect to the amplitude parameter.
By Proposition \ref{cor:ell_local_lipschitz_lambda}, the map
\[
W(\lambda):=U_1(\lambda)-U_2(\lambda)
\]
is locally Lipschitz from $(0,\infty)$ to $H^1(\Omega)$. By Theorem \ref{thm:ell_left_linearization}, the left weak $H^1$-derivative of $U_i$ at $\lambda$ is $V_i(\lambda)$. Hence the left derivative of $W$ is
\[
D^-W(\lambda)=V_1(\lambda)-V_2(\lambda).
\]
Lemma \ref{lem:ell_equal_linearized_velocities} shows that this vanishes for every $\lambda>0$. Therefore Lemma \ref{lem:ell_banach_constancy} implies that $W$ is constant on $(0,\infty)$.

Let $E_f$ be the harmonic extension of $f$ to $\Omega$. Since $f>0$ on $\p\Omega$, the strong maximum principle gives $E_f>0$ in $\Omega$. For $\lambda$ sufficiently large,
\[
\lambda E_f>\tphi_i \qquad \text{in }\Omega,\quad i=1,2,
\]
so the obstacle is inactive and
\[
U_i(\lambda)=\lambda E_f.
\]
Hence $W(\lambda)=0$ for all sufficiently large $\lambda$. Since $W$ is constant, it follows that
\[
U_1(\lambda)=U_2(\lambda)
\qquad \text{for every } \lambda>0.
\]
Finally, we claim that
\[
U_i(\lambda)\to \tphi_i
\qquad\text{in } C^1(\overline\Omega)
\]
as $\lambda\downarrow 0$. Indeed, the estimate \eqref{eq:ell_c1alpha_estimate} gives a uniform $C^{1,\alpha}(\overline\Omega)$ bound for $U_i(\lambda)$ on $\lambda\in(0,1]$. Hence every sequence $\lambda_j\downarrow 0$ has a subsequence, not relabeled, such that
\[
U_i(\lambda_j)\to \widetilde u_i
\qquad\text{in } C^1(\overline\Omega)
\]
for some $\widetilde u_i\in C^1(\overline\Omega)$. In particular, the convergence is also strong in $H^1(\Omega)$. Since $U_i(\lambda_j)\ge \tphi_i$ in $\Omega$ and the boundary data satisfy $U_i(\lambda_j)|_{\p\Omega}=\lambda_j f\to 0$, we obtain
\[
\widetilde u_i\ge \tphi_i,
\qquad
\widetilde u_i|_{\p\Omega}=0.
\]
Thus $\widetilde u_i$ belongs to the admissible class for the obstacle problem with obstacle $\tphi_i$ and zero boundary data.

We now verify the corresponding variational inequality. Let $v\in H^1(\Omega)$ satisfy $v\ge \tphi_i$ a.e. in $\Omega$ and $v|_{\p\Omega}=0$. Set
\[
v_j=v+\lambda_j E_f.
\]
In particular, $v_j\ge \tphi_i$ in $\Omega$ and
\[
v_j|_{\p\Omega}=\lambda_j f,
\]
so $v_j\in \mathcal K_{\tphi_i}^{\lambda_j f}$. Also, $v_j\to v$ in $H^1(\Omega)$ as $j\to\infty$. Applying \eqref{def-variational-intro} to $U_i(\lambda_j)$ with the test function $ v = v_j$ gives
\[
\int_\Omega \nabla U_i(\lambda_j)\cdot \nabla\bigl(v_j-U_i(\lambda_j)\bigr)\,dx\ge 0.
\]
Passing to the limit, using the strong $H^1(\Omega)$ convergence $U_i(\lambda_j)\to \widetilde u_i$ and $v_j\to v$, we obtain
\[
\int_\Omega \nabla \widetilde u_i\cdot \nabla\bigl(v-\widetilde u_i\bigr)\,dx\ge 0.
\]
Since this holds for every admissible $v$ with zero boundary data, $\widetilde u_i$ solves the obstacle problem with obstacle $\tphi_i$ and boundary value $0$.

We are now looking at the obstacle problem with obstacle $\tphi_i$ and zero outer boundary data; the claim is that its unique solution is precisely $\tphi_i$ itself.

It remains to identify this solution. The function $\tphi_i$ itself is admissible, because $\tphi_i\ge \tphi_i$ and $\tphi_i|_{\p\Omega}=0$. Moreover, for every admissible $v$ we have $v-\tphi_i\in H_0^1(\Omega)$ and $v-\tphi_i\ge 0$ a.e., so by integration by parts,
\[
\int_\Omega \nabla \tphi_i\cdot \nabla(v-\tphi_i)\,dx
=-\int_\Omega (\Delta \tphi_i)(v-\tphi_i)\,dx\ge 0,
\]
because $\Delta\tphi_i<0$ in $\Omega$. Thus $\tphi_i$ also satisfies the same variational inequality. By the uniqueness of the solution to the obstacle problem \eqref{def-variational-intro}, it follows that $\widetilde u_i=\tphi_i$.

We have shown that every sequence $\lambda_j\downarrow 0$ has a subsequence converging to $\tphi_i$ in $C^1(\overline\Omega)$. Since the family $\{U_i(\lambda):0<\lambda\le 1\}$ is precompact in $C^1(\overline\Omega)$ by \eqref{eq:ell_c1alpha_estimate}, this implies that the whole family converges:
\[
U_i(\lambda)\to \tphi_i
\qquad\text{in } C^1(\overline\Omega)
\quad\text{as }\lambda\downarrow 0.
\]
Therefore $\tphi_1=\tphi_2$.
\end{proof}


\section{Inverse Parabolic Obstacle Problem}
\label{sec:inverse_parabolic}
\subsection{Determining the coefficient}

We begin by isolating the recovery step for the time-independent coefficient $c(x)$ in the inverse parabolic obstacle problem. The point is that near the constant state $M_0$ the obstacle condition is inactive, so the nonlinear parabolic Dirichlet-to-Neumann map reduces to the Dirichlet-to-Neumann map for the linear equation $(\partial_t-c(x)\Delta)u(t,x)=0$.


\begin{lemma}
  \label{lm:coefficient_id}
  If $\Lambdapar_{c_1,\varphi_1}=\Lambdapar_{c_2,\varphi_2}$, then $c_1=c_2$.
\end{lemma}

\begin{proof}
For $i=1,2$, let $R_{c_i}$ be the Dirichlet-to-Neumann map for the linear parabolic problem
\begin{align}
  R_{c_i}: C^\infty_0((0,T)\times \p\Omega)&\to C^\infty((0,T)\times \p\Omega),\\
  R_{c_i}f&=\p_\nu u_i^f,
\end{align}
where $u_i^f=u_i^f(x,t)$ 
solves
\begin{equation}
  \label{eq:linear_heat_ci}
  \begin{cases}
    (\p_t-c_i(x)\Delta)u_i^f=0,&\text{in }(0,T)\times \Omega,\\
    u_i^f=f,&\text{on }(0,T)\times \p\Omega,\\
    u_i^f(0,\cdot)=0,&\text{in }\Omega.
  \end{cases}
\end{equation}
By the parabolic boundary regularity estimate \cite[Theorem 5.14]{Lieberman96}, there exists a constant $C_i$ such that
\begin{equation}
  \label{eq:heat_regularity}
  \norm{u_i^f}_{C^1([0,T]\times \Omega)}\le C_i \norm{f}_{C^3([0,T]\times \p\Omega)}.
\end{equation}
Let $C=\max\{C_1,C_2\}$. Choose $\varepsilon>0$ so small that $\varphi_i+\varepsilon<M_0$ in $\Omega$ for $i=1,2$, and define
\begin{equation}
  \BB_\varepsilon
  :=
  \left\{
    f\in C^\infty_0((0,T)\times \p\Omega):
    \norm{f}_{C^3([0,T]\times \p\Omega)}\le \frac{\varepsilon}{C}
  \right\}.
\end{equation}
Then for $f\in \BB_\varepsilon$, estimate \eqref{eq:heat_regularity} gives
\[
  \norm{u_i^f}_{C^0([0,T]\times \Omega)}\le \varepsilon,
\]
hence
\[
  M_0+u_i^f\ge M_0-\varepsilon>\varphi_i
  \qquad\text{in }[0,T]\times \Omega.
\]
Therefore the obstacle is inactive, and $M_0+u_i^f$ solves the parabolic obstacle problem \eqref{eq:par_obstacle_strong} with boundary data $M_0+f$. By uniqueness,
\[
u^{M_0+f}_{i,\varphi_i}=M_0+u_i^f.
\]
Consequently,
\begin{equation}
  \label{eq:small_linearization_heat}
  \Lambdapar_{c_i,\varphi_i}(M_0+f)
  =
  \p_\nu u^{M_0+f}_{i,\varphi_i}
  =
  \p_\nu u_i^f
  =
  R_{c_i}f
\end{equation}
for all $f\in \BB_\varepsilon$.

Since $R_{c_i}$ is linear, for arbitrary $f\in C^\infty_0((0,T)\times \p\Omega)$ we may choose
\[
C_f=\frac{\varepsilon}{C\norm{f}_{C^3([0,T]\times \p\Omega)}}
\]
so that $C_f f\in \BB_\varepsilon$. Then \eqref{eq:small_linearization_heat} yields
\[
R_{c_i}f=\frac1{C_f}R_{c_i}(C_f f)
=\frac1{C_f}\Lambdapar_{c_i,\varphi_i}(M_0+C_f f).
\]
Hence $\Lambdapar_{c_1,\varphi_1}=\Lambdapar_{c_2,\varphi_2}$ implies
\[
R_{c_1}=R_{c_2}.
\]

Now write $\rho_i=c_i^{-1}$. Equation \eqref{eq:linear_heat_ci} is equivalent to
\[
\rho_i(x)\p_t u(t,x)-\Delta u(t,x)=0.
\]
This is the class considered by Canuto--Kavian in their boundary determination result for heat equations \cite{CK01},
see also \cite{KKL}. Applying their uniqueness theorem to $\rho_1$ and $\rho_2$, we obtain $\rho_1=\rho_2$, and therefore $c_1=c_2$.
\end{proof}

From now on we assume $\Lambdapar_{c_1,\varphi_1}=\Lambdapar_{c_2,\varphi_2}$ and, by Lemma \ref{lm:coefficient_id}, write the common coefficient simply as $c(x)$.

\subsection{Additional regularity}
In this section we collect some regularity results for the parabolic obstacle problem for the reduction, which will allow us to reduce our analysis to the elliptic obstacle problem. 

As in the elliptic part, we begin by reducing our inverse problem to the case where the trace of obstacle vanishes on the boundary. Let $\psi=\psi(x)$ be defined as in \eqref{eq:elliptic}.
For $i=1,2$, set
\[
\widetilde{\varphi}_i=\varphi_i-\psi,\qquad
\widetilde{M_0}=M_0-\psi,\qquad
\widetilde{f}=f-\psi,\qquad
\widetilde{u}_{\varphi_i}^f=u_{\varphi_i}^f-\psi,
\]
where $u_{\varphi_i}^f(t,x)$ solves \eqref{eq:par_obstacle_strong} associated with the obstacle $\varphi_i$ and the boundary value $f$.
Then $\widetilde{u}_{\varphi_i}^f$ satisfies \eqref{eq:par_obstacle_strong} associated with obstacle $\widetilde{\varphi}_i$ and Dirichlet condition $\widetilde{f}$.
Moreover,
\[
\Lambdapar_{c,\widetilde{\varphi}_i}(\widetilde{f})
=
\p_\nu \widetilde{u}_{\varphi_i}^f
=
\Lambdapar_{c,\varphi_i}(f)-\p_\nu \psi.
\]
Thus $\Lambdapar_{c,\varphi_1}=\Lambdapar_{c,\varphi_2}$ implies $\Lambdapar_{c,\widetilde{\varphi}_1}=\Lambdapar_{c,\widetilde{\varphi}_2}$.

By the strong maximum principle \cite[Theorem 3.5]{gilbarg1998},
\[
\max_{x\in \Omega}\psi(x)=\max_{x\in \p\Omega}\varphi(x).
\]
Since $\varphi<M_0$ in $\Omega$, it follows that
\[
\widetilde{M_0}=M_0-\psi>\delta>0
\]
for some $\delta$ depending only on $M_0$ and $\varphi$.

The next lemma is the parabolic analogue of the boundary collar regularity used in the elliptic section.

\begin{lemma}
  \label{lm:additional_regularity}
  Let $\tphi,\tM$ be defined as above. Then for any $T>0$ and $f\in \cApar_{\tphi,\tM}$, there exists a neighborhood $U$ of $\p\Omega$ such that $u_{\tphi}^f|_{[0,T]\times U}$ is smooth. Moreover,
  \[
  \nabla u_{\tphi}^f\in H^1([0,T];(H^1(\Omega)^n)^\ast).
  \]
\end{lemma}

To prove Lemma \ref{lm:additional_regularity}, we use two auxiliary lemmas.

\begin{lemma}
  \label{lm:interior_regularity}
  Let $c\in C^\infty(\overline\Omega)$ and $U\Subset \widetilde U\subset \Omega$. Let $0<\widetilde T_1<T_1<T_2<\widetilde T_2$. Suppose
  \[
  u\in C^0((\widetilde T_1,\widetilde T_2)\times \widetilde U)
  \]
  satisfies
  \[
  (\p_t-c(x)\Delta)u(t,x)=0
  \qquad\text{in }(\widetilde T_1,\widetilde T_2)\times \widetilde U
  \]
  in the weak sense. Then $u\in C^\infty([T_1,T_2]\times \overline U)$. 
  Moreover, for each $k\in \mathbb N$ there is a constant $C_k$ such that
  \begin{equation}
    \label{eq:lm_3_1}
    \norm{u}_{C^k([T_1,T_2]\times \overline U)}
    \le C_k \norm{u}_{C^0((\widetilde T_1,\widetilde T_2)\times \widetilde U)}.
  \end{equation}
\end{lemma}


\begin{proof}
To simplify notation, define
\[
Q_r(t,x):=(t-r^2/2,t+r^2/2)\times B(x,r).
\]
For any $(t,x)\in [T_1,T_2]\times U$, choose $\delta>0$ such that
\[
Q_{2\delta}(t,x)\subset (\widetilde T_1,\widetilde T_2)\times \widetilde U.
\]
Let $\eta_\varepsilon(t)$ and $\xi_\varepsilon(x)$ be standard time and space mollifiers, and set
\[
u_\varepsilon(t,x)
=
\int_{-\varepsilon}^{\varepsilon}
\int_{B(0,\varepsilon)}
\eta_\varepsilon(\tau)\xi_\varepsilon(y)u(t-\tau,x-y)\,dy\,d\tau.
\]
For $\varepsilon$ sufficiently small, $u_\varepsilon\in C^\infty(Q_{2\delta}(t,x))$ and
\[
\norm{u_\varepsilon-u}_{C^0(Q_{2\delta}(t,x))}\to 0
\]
as $\varepsilon\to 0$; see \cite[Appendix C, Theorem 7]{evans}. Moreover,
\begin{align*}
f_\varepsilon(t,x) &:= (\p_t-c(x)\Delta)u_\varepsilon(t,x) \\
&= \int_{-\varepsilon}^{\varepsilon}\int_{B(0,\varepsilon)} \eta_\varepsilon(\tau)\xi_\varepsilon(y)\big(c(x-y)-c(x)\big)\,\Delta u(t-\tau,x-y)\,dy\,d\tau.
\end{align*}
Since $c$ is smooth, $\norm{f_\varepsilon}_{C^k(Q_{2\delta}(t,x))}\to 0$ for each $k\ge 0$. By interior parabolic regularity \cite[Theorem 8.12.1]{Krylov96},
\[
\norm{u_\varepsilon}_{C^1(Q_\delta(t,x))}
\le C\left(
\norm{f_\varepsilon}_{C^0(Q_{2\delta}(t,x))}
+
\norm{u_\varepsilon}_{C^0(Q_{2\delta}(t,x))}
\right).
\]
Hence $\{u_\varepsilon\}$ is Cauchy in $C^1(Q_\delta(t,x))$, so $u\in C^1(Q_\delta(t,x))$ and
\[
\norm{u}_{C^1(Q_\delta(t,x))}
\le C\norm{u}_{C^0(Q_{2\delta}(t,x))}.
\]
By compactness of $[T_1,T_2]\times \overline U$, this yields \eqref{eq:lm_3_1} for $k=1$. Iterating the argument for higher derivatives proves smoothness and the general estimate.
\end{proof}

\begin{lemma}
  \label{lm:transient_smooth}
  Let $f\in \cApar_{\tphi,\tM}$ and let $u_{\tphi}^f$ solve \eqref{eq:par_obstacle_strong} associated with obstacle $\tphi$ and Dirichlet condition $f$. Then there exists $\tau>0$ such that
  \[
  u_{\tphi}^f|_{[0,\tau]\times \Omega}\in C^\infty([0,\tau]\times \Omega).
  \]
\end{lemma}

\begin{proof}
Let $v$ solve
\[
\begin{cases}
(\p_t-c(x)\Delta)v(t,x)=0,&\text{in }[0,T]\times \Omega,\\
v|_{[0,T]\times \p\Omega}=f,\qquad v(0,\cdot)=\tM.
\end{cases}
\]
Choose $\delta>0$ such that $\tM>\tphi+\delta$ in $\Omega$. By the boundary Hölder estimate \cite[Theorem 5.14]{Lieberman96}, $v$ stays within $\delta$ of $\tM$ for small positive time. Hence there exists $\tau>0$ such that
\[
v(t,x)>\tphi(x)
\qquad\text{in }[0,\tau]\times \Omega.
\]
Therefore the obstacle is inactive on $[0,\tau]\times \Omega$, and uniqueness for \eqref{eq:par_obstacle_strong} gives $u_{\tphi}^f=v$ there. Since $v$ is a smooth solution of a linear parabolic equation, the claim follows.
\end{proof}

\begin{proof}[Proof of Lemma \ref{lm:additional_regularity}]
Since $u_{\tphi}^f\in W^{1,2}_p([0,T]\times \Omega)$ for every $1<p<\infty$, Sobolev embedding yields
\[
u_{\tphi}^f\in C^0([0,T]\times \Omega).
\]
Also, by \cite[Theorem 8.2]{Friedman82},
\[
\nabla u_{\tphi}^f\in L^2([0,T];L^2(\Omega)^n)\subset L^2([0,T];(H^1(\Omega)^n)^\ast),
\]
and since $\p_t u_{\tphi}^f\in L^2([0,T];L^2(\Omega))$, we have
\[
\nabla \p_t u_{\tphi}^f\in L^2([0,T];H^{-1}(\Omega)^n)\subset L^2([0,T];(H^1(\Omega)^n)^\ast).
\]
Thus
\[
\nabla u_{\tphi}^f\in H^1([0,T];(H^1(\Omega)^n)^\ast).
\]

It remains to prove smoothness near the boundary. Since $\tM>0$ and $f>0$, choose $\delta>0$ such that
\[
\tM>\delta>0
\qquad\text{and}\qquad
f(t,x)>\delta>0
\quad\text{on }[0,T]\times \p\Omega.
\]
Because $u_{\tphi}^f(t,x)$ is continuous and satisfies
\[
(\p_t-c(x)\Delta)u_{\tphi}^f(t,x)\ge 0
\]
in the weak sense, the parabolic maximum principle for supersolutions (see e.g. \cite[Corollary 7.4]{Lieberman96}) implies
\[
u_{\tphi}^f\ge \min\{\tM,f\}>\delta
\qquad\text{in }[0,T]\times \Omega.
\]
Define the collar neighborhood
\begin{equation}
  \label{eq:def_collar_neighborhood}
  U_\eta:=\{x\in \Omega:\dist(x,\p\Omega)<\eta\}.
\end{equation}
Since $\tphi|_{\p\Omega}=0$ and $\tphi$ is continuous, there exists $\widetilde \eta>0$ such that
\[
\tphi(x)<\delta
\qquad\text{for }x\in U_{\widetilde \eta}.
\]
Hence $u_{\tphi}^f>\tphi$ in $[0,T]\times U_{\widetilde \eta}$, so the obstacle is inactive there and
\[
(\p_t-c(x)\Delta)u_{\tphi}^f(t,x)=0
\qquad\text{in }[0,T]\times U_{\widetilde \eta}.
\]

By \cite[Lemma 14.16]{gilbarg1998}, we may choose $\eta\in (0,\widetilde \eta)$ such that $U_\eta$ is a smooth domain. Let
\[
\Upsilon_\eta:=\{x\in \Omega:\dist(x,\p\Omega)=\eta\}.
\]
Lemma \ref{lm:transient_smooth} gives smoothness on $[0,\tau]\times \Upsilon_\eta$ for some $\tau>0$, and Lemma \ref{lm:interior_regularity} gives smoothness on $[\tau,T]\times \Upsilon_\eta$. Thus the trace
\[
g:=u_{\tphi}^f|_{[0,T]\times \Upsilon_\eta}
\]
is smooth. Therefore $u_{\tphi}^f(t,x)$ solves the classical linear boundary value problem
\[
\begin{cases}
(\p_t-c(x)\Delta)u_{\tphi}^f(t,x)=0,&\text{in }[0,T]\times U_\eta,\\
u_{\tphi}^f|_{[0,T]\times \Upsilon_\eta}=g,\qquad
u_{\tphi}^f|_{[0,T]\times \p\Omega}=f,\\
u_{\tphi}^f(0,\cdot)|_{U_\eta}=\tM|_{U_\eta},
\end{cases}
\]
and hence
\[
u_{\tphi}^f\in C^\infty([0,T]\times \overline{U_\eta}).
\]
This proves the lemma.
\end{proof}



\subsection{Reduction to the inverse elliptic obstacle problem}

We now pass from the parabolic obstacle problem to the stationary elliptic one by showing that solutions with eventually constant boundary data converge to the elliptic obstacle solution. Our main result in this section is the following proposition.




\begin{proposition}
  \label{thm:reduction}
  Let $\tphi_1,\tphi_2\in C^{2,1}(\overline\Omega)$, satisfy
  \begin{equation}
      \tphi_1|_{\p\Omega} = \tphi_2|_{\p\Omega} = 0,\qquad \Delta \tphi_1 < 0,\ \Delta \tphi_2 <0 \text{ in }\Omega.
  \end{equation}
  Let $c\in C^\infty(\overline{\Omega})$ be positive.
     Suppose that
  \[
  \Lambdapar_{c,\tphi_1}(f)=\Lambdapar_{c,\tphi_2}(f)\text{  for all } f\in \cApar_{\tphi_1,\tM}.
  \]
  Then
  \[
  \Theta_{\tphi_1}(g)=\Theta_{\tphi_2}(g)
  \text{  for all }  g > 0.
  \]
\end{proposition}


We argue through the variational formulation.
To simplify notations, we write
\begin{equation}
  \label{eq:def_a}
  a(u,v):=\int_\Omega \nabla u(x)\cdot \nabla v(x)\,dx,
\end{equation}
and denote by
\[
A(u):=a(u,u)=\int_\Omega |\nabla u|^2\,dx
\]
the associated quadratic form.
Then we have the following equivalent representation for the variational form \eqref{eq:par_obstacle_weak}.
\begin{lemma}
  \label{lm:par_weak_form}
  Let $u$ solve \eqref{eq:par_obstacle_strong} with obstacle $\tphi$ and boundary condition $f\in \cApar_{\tphi,\widetilde M_0}$. Then for any $v\in \cKpar_{\tphi,f}$ and for a.e. $t\ge 0$, there holds
  \begin{equation}
    \label{eq:par_weak_form}
    \langle c^{-1}\p_t u(t),\,v(t)-u(t)\rangle_{L^2(\Omega)}+a(u(t),\,v(t)-u(t))\ge 0.
  \end{equation}
\end{lemma}

\begin{proof}
Let
\[
D:=\{(t,x)\in [0,\infty)\times \Omega:
(\p_t-c\Delta)u<0
\text{ or }
(\p_tu-c\Delta u)\cdot (u-\tphi)\neq 0\},
\]
and
\[
D(t):=\{x\in \Omega:(t,x)\in D\}.
\]
Since $u$ solves \eqref{eq:par_obstacle_strong}, the set $D$ has measure zero. Fubini's theorem implies that for a.e. $t\ge 0$,
\[
|D(t)|=0.
\]
For such $t$, define the contact set
\[
E(t):=\{x\in \Omega: u(t,x)=\tphi(x)\}.
\]
Let $v\in K$. On $E(t)\setminus D(t)$ we have $v(t)-u(t)\ge 0$ and $c^{-1}(\p_tu-c\Delta u)\ge 0$, so
\[
\int_{E(t)\setminus D(t)} c^{-1}(\p_tu-c\Delta u)\cdot (v-u)\,dx\ge 0.
\]
On $(\Omega\setminus D(t))\setminus E(t)$ the obstacle is inactive, hence $(\p_t-c\Delta)u=0$ there and therefore
\[
\int_{(\Omega\setminus D(t))\setminus E(t)} c^{-1}(\p_tu-c\Delta u)\cdot (v-u)\,dx=0.
\]
Summing these identities gives
\[
\int_\Omega c^{-1}\p_tu(v-u)\,dx+\int_\Omega \nabla u\cdot \nabla(v-u)\,dx\ge 0,
\]
which is exactly \eqref{eq:par_weak_form}.
\end{proof}


We say that a function $h:\RR_+\to \RR$ converges essentially to $0$, denoted by
\[
\limess_{t\to\infty} h(t)=0,
\]
if there exists a set $N\subset \RR_+$ of measure zero such that
\[
\lim_{\substack{t\to\infty\\ t\notin N}} h(t)=0.
\]


\begin{lemma}
  \label{lm:H_1_asymptotic}
  Let $f\in \cApar_{\tphi,\tM}$ and let $g\in C^\infty(\p\Omega)$ satisfy $g>0$. Suppose $f(t,\cdot)=g$ for all $t\ge t_0>0$. Let $v_{\tphi}^g(t,\cdotp)$ solve \eqref{def-obstacle-intro} associated with obstacle $\tphi$ and boundary condition $g$. Then
  \[
  \limess_{t\to \infty}\norm{\p_tu_{\tphi}^f(t)}_{L^2(\Omega)}^2=0
  \]
  and
  \[
  \limess_{t\to \infty}\norm{u_{\tphi}^f(t)-v_{\tphi}^g}_{H^1(\Omega)}=0.
  \]
\end{lemma}

\begin{proof}
For $t\ge t_0$, define
\[
A[u_{\tphi}^f](t):=a(u_{\tphi}^f(t),u_{\tphi}^f(t))
=\int_\Omega |\nabla u_{\tphi}^f(t,x)|^2\,dx.
\]
Since $u_{\tphi}^f\in L^2((0,T);H^1(\Omega))$ and $\p_tu_{\tphi}^f\in L^2((0,T);L^2(\Omega))$, we have $A[u_{\tphi}^f] \in W^{1,1}(0,T)$ and
\[
\p_tA[u_{\tphi}^f](t)=2a(\p_tu_{\tphi}^f(t),u_{\tphi}^f(t))
\]
for a.e. $t$; see \cite[Corollary 1.4.39]{CH98}.

For any fixed $t>t_0$, we can find $\delta>0$ such that $t>t_0+2\delta$.. Since $f(t) = g$ for all $t > t_0$, for every sufficiently small $h$ there exists $v_h^\pm\in \cKpar_{\tphi,f}$ such that
\[
v_h^\pm(\tau):=u_{\tphi}^f(\tau\pm h),
\qquad \tau\in (t-\delta,t+\delta).
\]
Inserting $v_h^\pm$ into \eqref{eq:par_weak_form}, we obtain
\[
\langle c^{-1}\p_tu_{\tphi}^f(\tau),u_{\tphi}^f(\tau\pm h)-u_{\tphi}^f(\tau)\rangle_{L^2(\Omega)}
+
a(u_{\tphi}^f(\tau),u_{\tphi}^f(\tau\pm h)-u_{\tphi}^f(\tau))
\ge 0
\]
for a.e. $\tau\in (t-\delta,t+\delta)$. Dividing by $h$ and letting $h\to 0^\pm$, again using \cite[Corollary 1.4.39]{CH98}, yields
\[
\langle c^{-1}\p_tu_{\tphi}^f(\tau),\p_tu_{\tphi}^f(\tau)\rangle_{L^2(\Omega)}
=
-\frac12 \p_tA[u_{\tphi}^f](\tau)
\]
for a.e. $\tau>t_0$. In particular, $A[u_{\tphi}^f]$ is nonincreasing on $(t_0,\infty)$, and for any $T>t>t_0$,
\begin{equation}
  \label{eq:energy_decay_parabolic}
  \int_t^T \langle c^{-1}\p_tu_{\tphi}^f(\tau),\p_tu_{\tphi}^f(\tau)\rangle_{L^2(\Omega)}\,d\tau
  =
  \frac12\left(A[u_{\tphi}^f](t)-A[u_{\tphi}^f](T)\right).
\end{equation}
Since $c$ is bounded away from zero and infinity, this implies
\[
\int_t^\infty \norm{\p_tu_{\tphi}^f(\tau)}_{L^2(\Omega)}^2\,d\tau<\infty
\]
for every $t>t_0$. Therefore
\[
\limess_{t\to\infty}\norm{\p_tu_{\tphi}^f(t)}_{L^2(\Omega)}^2=0.
\]

Now let $v_{\tphi}^g$ be the elliptic obstacle solution with boundary data $g$. For $t>t_0$, we have $u_{\tphi}^f(t)\in K_{\tphi}^g$, while $v_{\tphi}^g$ minimizes $A$ on $K_{\tphi}^g$. Hence
\[
A(v_{\tphi}^g)\le A(u_{\tphi}^f(t)).
\]
Moreover,
\begin{align*}
0
&\le A(v_{\tphi}^g-u_{\tphi}^f(t))\\
&=A(v_{\tphi}^g)-A(u_{\tphi}^f(t))-2a(v_{\tphi}^g-u_{\tphi}^f(t),u_{\tphi}^f(t)),
\end{align*}
so
\begin{equation}
  \label{eq:parabolic_compare1}
  A(u_{\tphi}^f(t))-A(v_{\tphi}^g)\le -2a(u_{\tphi}^f(t),v_{\tphi}^g-u_{\tphi}^f(t)).
\end{equation}
Choose $v\in K$ such that $v(t,\cdot)=v_{\tphi}^g$ for all $t>t_0$. Then Lemma \ref{lm:par_weak_form} gives
\[
-a(u_{\tphi}^f(t),v_{\tphi}^g-u_{\tphi}^f(t))
\le
\langle c^{-1}\p_tu_{\tphi}^f(t),v_{\tphi}^g-u_{\tphi}^f(t)\rangle_{L^2(\Omega)},
\]
and therefore
\begin{equation}
  \label{eq:parabolic_compare2}
  A(u_{\tphi}^f(t))-A(v_{\tphi}^g)
  \le
  2\langle c^{-1}\p_tu_{\tphi}^f(t),v_{\tphi}^g-u_{\tphi}^f(t)\rangle_{L^2(\Omega)}.
\end{equation}
On the other hand, since $v_{\tphi}^g$ solves the elliptic variational inequality \eqref{def-variational-intro},
\[
A(v_{\tphi}^g-u_{\tphi}^f(t))\le A(u_{\tphi}^f(t))-A(v_{\tphi}^g).
\]
Because $v_{\tphi}^g-u_{\tphi}^f(t)\in H_0^1(\Omega)$, coercivity and Poincar\'e's inequality imply
\begin{equation}
  \label{eq:parabolic_compare3}
  \norm{v_{\tphi}^g-u_{\tphi}^f(t)}_{H^1(\Omega)}
  \lesssim
  \bigl(A(u_{\tphi}^f(t))-A(v_{\tphi}^g)\bigr)^{1/2}.
\end{equation}
Combining \eqref{eq:parabolic_compare2} and \eqref{eq:parabolic_compare3}, we obtain
\[
\bigl(A(u_{\tphi}^f(t))-A(v_{\tphi}^g)\bigr)^{1/2}
\lesssim
\norm{\p_tu_{\tphi}^f(t)}_{L^2(\Omega)}.
\]
Since the right-hand side converges essentially to $0$, we conclude that
\[
A(u_{\tphi}^f(t))-A(v_{\tphi}^g)\to 0
\]
essentially as $t\to \infty$. Using \eqref{eq:parabolic_compare3} once more proves
\[
\limess_{t\to\infty}\norm{u_{\tphi}^f(t)-v_{\tphi}^g}_{H^1(\Omega)}=0.
\]
\end{proof}

\begin{lemma}
  \label{lm:N_trace_asymptotic}
  Let $f\in \cApar_{\tphi,\tM}$ and let $g\in C^\infty(\p\Omega)$ satisfy $g>0$. Suppose $f(t,\cdot)=g$ for $t\ge t_0>0$. Let $v_{\tphi}^g$ solve \eqref{def-obstacle-intro} associated with obstacle $\tphi$ and boundary condition $g$, then
  \[
  \limess_{t\to\infty}
  \norm{\p_\nu(u_{\tphi}^f(t)-v_{\tphi}^g)}_{H^{-1/2}(\p\Omega)}
  =0.
  \]
\end{lemma}

\begin{proof}
Let $U\subset \RR^n$ be a smooth domain. Since the trace map $H^1(U)\to H^{1/2}(\p U)$ is continuous and surjective, for any $w\in H^1(U)$ with $\Delta w\in L^2(U)$,
\[
\norm{\p_\nu w}_{H^{-1/2}(\p U)}
\lesssim
\norm{w}_{H^1(U)}+\norm{\Delta w}_{L^2(U)}.
\]
By the maximum principle,
\[
v_{\tphi}^g\ge \min_{x\in \p\Omega}g(x)>0
\qquad\text{in }\Omega.
\]
As in the proof of Lemma \ref{lm:additional_regularity}, there exists a collar neighborhood $U_\eta$ of $\p\Omega$ such that
\[
\Delta v_{\tphi}^g=0 \quad\text{in }U_\eta,
\qquad
(\p_t-c\Delta)u_{\tphi}^f=0 \quad\text{in }[0,\infty)\times U_\eta.
\]
Hence in $U_\eta$,
\[
\Delta u_{\tphi}^f(t)=c^{-1}\p_tu_{\tphi}^f(t).
\]
Applying the trace estimate on $U_\eta$ to
\[
w(t)=u_{\tphi}^f(t)-v_{\tphi}^g,
\]
and using Lemma \ref{lm:H_1_asymptotic}, we obtain
\begin{align*}
\limess_{t\to\infty}
\norm{\p_\nu w(t)}_{H^{-1/2}(\p\Omega)}
&\lesssim
\limess_{t\to\infty}
\left(
\norm{w(t)}_{H^1(U_\eta)}
+
\norm{\Delta w(t)}_{L^2(U_\eta)}
\right)\\
&\le
\limess_{t\to\infty}
\left(
\norm{w(t)}_{H^1(\Omega)}
+
\norm{\p_tu_{\tphi}^f(t)}_{L^2(\Omega)}
\right)
=0.
\end{align*}
\end{proof}
Finally, we note that Proposition \ref{thm:reduction} follows immediately from Lemma \ref{lm:N_trace_asymptotic}. And Theorem \ref{thm:main} is proved by combining Proposition \ref{thm:reduction} and Theorem \ref{thm:ellptic_obstacle}.




\appendix

\subsection*{Acknowledgements}
The authors would like to thank Xavier Fern\'andez-Real for valuable comments.
C.C. was supported by NSTC grant  113-2115-M-A49-018-MY3.
M.L. was partially supported by a AdG project 101097198 of the European Research Council, Centre of Excellence of Research Council of Finland and the FAME flagship of the Research Council of Finland (grant 359186).
L.O. was supported by the European Research Council of the European Union, grant 101086697 (LoCal), and the Research Council of Finland, grants 359182, 347715 and 353096.
Z.Z. was supported by the Finnish Ministry of Education and Culture’s Pilot for Doctoral Programmes (Pilot project Mathematics of Sensing, Imaging and Modelling).
Views and opinions expressed are those of the authors only and do not necessarily reflect those of the funding agencies or the EU. Neither the European Union nor the granting authority can be held responsible for them.

\bibliography{reference}
\bibliographystyle{plain}

\end{document}